\documentclass[11pt, english]{article}
\usepackage[T1]{fontenc}
\usepackage[latin9]{inputenc}
\usepackage{textcomp}
\usepackage{amsmath}
\usepackage{amssymb}
\usepackage{stmaryrd}
\usepackage{babel}
\usepackage{algorithm}
\usepackage{algpseudocode}

\usepackage{amsmath,amssymb,amsthm,mathrsfs,hyperref,color}

\DeclareMathAlphabet{\mathpzc}{OT1}{pzc}{m}{it}
\usepackage[nameinlink]{cleveref}
\hypersetup{colorlinks={true},linkcolor={blue},citecolor=blue}

\hypersetup{
    colorlinks=true,
    linkcolor=blue!60!black,
    citecolor=blue!60!black,
    urlcolor=blue!60!black
}

\usepackage{mathtools}
\usepackage[x11names]{xcolor}
\newtagform{blue}{\color{blue}(}{)}

\usepackage{aliascnt}

\theoremstyle{plain}
\newtheorem{theorem}{Theorem}[section]

\newaliascnt{proposition}{theorem}

\aliascntresetthe{proposition}

\newaliascnt{lemma}{theorem}
\newtheorem{lemma}[lemma]{Lemma}
\aliascntresetthe{lemma}

\newaliascnt{corollary}{theorem}

\aliascntresetthe{corollary}

\newaliascnt{assumption}{theorem}
\newtheorem{assumption}[assumption]{Assumption}
\aliascntresetthe{assumption}

\theoremstyle{definition}
\newaliascnt{definition}{theorem}
\newtheorem{definition}[definition]{Definition}
\aliascntresetthe{definition}

\newaliascnt{remark}{theorem}
\newtheorem{remark}[remark]{Remark}
\aliascntresetthe{remark}

\newaliascnt{example}{theorem}

\aliascntresetthe{example}

\DeclareMathOperator*{\argmin}{argmin}

\numberwithin{equation}{section}
\usepackage{authblk}

\title{\bf State-Dependent Delays in Optimal Control and Hamilton--Jacobi Equations}

\author[1]{Yiming Jiang}
\author[2]{Cristian Mendico}
\author[1]{Yawei Wei}
\author[3]{Fei Zeng}
\author[3]{Zimeng Zheng}
\affil[1]{ School of Mathematical Sciences and LPMC\\ Nankai University\\ Tianjin 300071 China\\ \href{mailto:ymjiangnk@nankai.edu.cn}{ymjiangnk@nankai.edu.cn}, \\ \href{mailto:weiyawei@nankai.edu.cn}{weiyawei@nankai.edu.cn}}
\affil[2]{Institut de Math\'ematique de Bourgogne - UMR 5584 CNRS, Universit\'e Bourgogne Europe\\ \href{mailto:cristian.mendico@u-bourgogne.fr}{cristian.mendico@u-bourgogne.fr}}
\affil[3]{ School of Mathematical Sciences\\ Nankai University\\ Tianjin 300071 China \\ \href{mailto:
faith@mail.nankai.edu.cn}{
faith@mail.nankai.edu.cn},\\  \href{mailto:
1120220038@mail.nankai.edu.cn}{
1120220038@mail.nankai.edu.cn}}
\date{}

\begin{document}
\maketitle

\begin{abstract}
	We develop a Hamilton--Jacobi theory for finite-horizon optimal control problems governed by state-dependent delay equations. In this setting, the delayed-time map depends on the controlled trajectory itself, so that both the state and the point at which the past state is evaluated vary simultaneously. The natural state variable is therefore the entire history, and the value functional is defined on a space of Lipschitz histories. Under suitable growth, Lipschitz, and monotonicity assumptions, we establish the dynamic programming principle and characterize the value functional as the unique viscosity solution of the associated Hamilton--Jacobi equation with co-invariant derivatives. To accommodate Lipschitz histories, we introduce a notion of viscosity solution based on finite-dimensional projections generated by polygonal extensions, and prove a comparison principle in the resulting class of functionals. Under additional regularity assumptions, we derive a Pontryagin minimum principle whose adjoint equation contains advanced terms induced by the state dependence of the delay. We also obtain a generalized transversality relation formulated through the delay superdifferential of the value functional. Finally, we establish semiconcavity in the history variable and a joint semiconcavity estimate on an appropriate solution manifold.

		\vspace{0.25cm}
		\\
		\noindent\textbf{2020 AMS:} 49L25, 49L20; 49K15; 34K35; 93C23
	\end{abstract}

\section{Introduction}
\subsection*{Our model and motivations} 

Delay differential equations provide a natural framework for systems whose evolution depends on both the current state and information inherited from the past. In many applications, however, the relevant past time is not fixed a priori. Rather, it changes with the evolving configuration of the system. This leads to \emph{state-dependent delay equations}, for which the delayed time is itself coupled to the unknown trajectory. Such equations arise in population dynamics \cite{MR3369216,MAHAFFY1998135}, traffic flow models \cite{RevModPhys.73.1067,9793393,martinovich_introducing_2025}, and, more recently, machine learning and neural networks \cite{10.1063/5.0325998,pmlr-v255-monsel24a,PhysRevE.111.035313}. In these settings, control variables may represent, for example, external interventions, flow regulations, or trainable parameters. 

This motivates the study of optimal control problems in which the dynamics are governed by a state-dependent delay. We consider a finite-horizon control problem governed by \[ \dot{x}(s) = f\bigl(s,x(s),x(s-\tau(s,x(s))),\alpha(s)\bigr), \qquad s\in(t,T), \] where \(\alpha\) is a measurable control and \( \tau:[t_0,T]\times\mathbb{R}^n\longrightarrow[0,\bar\tau] \) depends on both time and the current state. Given an initial history \( w\in \operatorname{Lip}([-\bar\tau,0];\mathbb{R}^n), \) the state is required to satisfy \[ x(t+\theta)=w(\theta), \qquad \theta\in[-\bar\tau,0]. \] Under appropriate structural assumptions, this determines a unique trajectory. With each initial pair \((t,w)\) and admissible control \(\alpha\), we associate a cost functional \(J(t,w,\alpha)\) and define the value functional by \[ \rho(t,w)=\inf_{\alpha}J(t,w,\alpha). \] Thus, unlike in finite-dimensional control problems, the natural state variable for dynamic programming is the entire history \(w\), and the corresponding Hamilton--Jacobi equation is posed on a function space.

The dependence of the delay on the state creates difficulties that are absent for prescribed delays. Indeed, the delayed-time map \[ D_x(s):=s-\tau(s,x(s)) \] depends on the trajectory itself. Consequently, when two trajectories are compared, both the trajectories and the points at which they are evaluated vary simultaneously: \( x_1\bigl(D_{x_1}(s)\bigr)\) and \(x_2\bigl(D_{x_2}(s)\bigr). \) This coupling affects several levels of the analysis. It complicates stability estimates for the state equation, introduces additional terms in the linearized and adjoint equations, and produces nonlinear second-order effects in the regularity analysis of the value functional. Moreover, the natural phase space consists of Lipschitz histories, whose behavior under time translation must be incorporated into the formulation of the Hamilton--Jacobi equation. 

The aim of this paper is to develop a unified optimal control and Hamilton--Jacobi framework capable of addressing these difficulties. In particular, we seek to connect the dynamic programming characterization of the value functional with first-order optimality conditions and regularity properties in a setting where the delay varies along the controlled trajectory. The main results and the technical constructions used to achieve this goal are described next.

\subsection*{Main results} 

The analysis is organized around three closely related objectives: the Hamilton--Jacobi characterization of the value functional, the derivation of first-order optimality and sensitivity relations, and the study of its second-order regularity. A recurring difficulty is that the delayed-time map \[ D_x(s):=s-\tau(s,x(s)) \] depends on the trajectory itself. Thus, when two trajectories are compared, one must control \[ x_1\bigl(D_{x_1}(s)\bigr)-x_2\bigl(D_{x_2}(s)\bigr), \] where both the trajectories and their evaluation points vary. This feature distinguishes the present problem from control systems with prescribed delays and requires estimates that are stable under simultaneous perturbations of the state and of the delayed-time map. 

Our first contribution is a dynamic programming and viscosity solution theory on the space of Lipschitz histories. We establish the dynamic programming principle in \Cref{thm:DPP} and formulate the associated Hamilton--Jacobi equation using co-invariant derivatives. Following the finite-dimensional projection strategy of \cite{MR4264642}, we test history-dependent functionals through functions of the current time and state. A direct use of constant extensions, as in the prescribed-delay setting, does not retain the information on the local variation of a Lipschitz history that is needed in our analysis. We therefore introduce polygonal extensions, obtained by joining the endpoint of the initial history to a perturbed terminal state. These extensions yield finite-dimensional projections adapted to the Lipschitz phase space and provide the basis for our definition of viscosity solution. Within this framework, we prove that the value functional is a viscosity solution of the Hamilton--Jacobi equation (\Cref{thm:value_viscosity}).

We then establish a comparison principle (\Cref{thm:comparison}) by a doubling-of-variables argument adapted to the projected functionals. A central ingredient is the stability of delayed polygonal states under changes of both the projection time and its endpoint. The comparison principle yields the uniqueness of the value functional in the class \(\Phi\) introduced below. In particular, the dynamic programming and viscosity approaches provide equivalent characterizations of the value functional in this state-dependent delay setting.

Our second contribution concerns first-order optimality conditions and their connection with the value functional. Under additional differentiability assumptions, we study localized perturbations of the initial history, measured in the uniform norm, and derive the corresponding linearized state equation. Because the delay depends on the current state, the linearization contains an additional contribution involving the derivative of \(\tau\) and the derivative of the trajectory at the delayed time. This term has no counterpart when the delay is prescribed and propagates into the dual equation. Using this linearization, we derive a Pontryagin minimum principle (\Cref{thm:PMP}). The associated adjoint equation is an advanced equation: its value at a given time depends on the adjoint state at the future time determined by the inverse of the delayed-time map. The adjoint equation also contains the additional terms generated by the state dependence of the delay. We then connect the maximum-principle and dynamic-programming descriptions of optimality. More precisely, along an optimal trajectory, the pair consisting of the negative minimized Hamiltonian and the adjoint state belongs to the delay superdifferential of the value functional (\Cref{thm:generalized_transversality}). This provides a generalized transversality, or sensitivity, relation in the Lipschitz history space.

Our third contribution is a second-order regularity theory for the value functional. Under suitable convexity and regularity assumptions on the control dependence of the Hamiltonian, we first prove that every initial datum admits an optimal control with a Lipschitz representative. This regularity is essential for controlling the second-order variation of the composite delayed state \[ x\bigl(D_x(s)\bigr) = x\bigl(s-\tau(s,x(s))\bigr). \] For histories that are piecewise \(C^{1,1}\), with a uniformly bounded number of nonsmooth points, we establish a second-order estimate for the corresponding trajectories.

 As a consequence, we prove that the value functional is locally semiconcave in the history variable with respect to the \(H^1\) norm. Joint semiconcavity in time and history requires an additional compatibility condition at the junction between the initial history and the forward trajectory. Indeed, a time perturbation shifts this junction and must preserve the \(H^1\) control of the resulting history. To encode this compatibility, we introduce a \(C^{1,1}\) solution manifold defined by matching the derivative of the history at its endpoint with the controlled vector field. On this manifold, we derive the joint semiconcavity estimate of \Cref{thm:time_semiconcavity_value}. 
 
 Taken together, these results connect three complementary descriptions of the control problem. The dynamic programming principle yields the Hamilton--Jacobi characterization of the value functional; the linearized and adjoint equations yield first-order optimality and sensitivity relations; and the second-order trajectory estimates yield semiconcavity. The common analytical mechanism is the control of delayed states under simultaneous perturbations of the trajectory and of the state-dependent delayed-time map.

\subsection*{Historical Background}
Differential equations with constant or distributed delays have been extensively studied; we refer to \cite{MR1783365} for a comprehensive treatment. In these settings, the initial history may generally be taken to be merely piecewise continuous. 

State-dependent delays, by contrast, require a more delicate analysis. Existence of solutions for continuous initial histories was established in \cite{MR150421}, while uniqueness was obtained under the additional assumption that the initial history is Lipschitz continuous. A solution-manifold framework for \(C^1\) histories was introduced in \cite{MR2019242}; see also \cite{MR2457636}. More recently, an \(H^1\)-theory encompassing Lipschitz histories was developed in \cite{MR4785300}. This provides the basic functional-analytic setting for the control problem considered in this paper. 

In optimal control problems with delays, the associated Hamilton--Jacobi equation is naturally posed on an infinite-dimensional space of histories. Several approaches have been developed to study such equations, including viscosity-solution methods \cite{MR690039,MR732102,MR794776}, minimax methods originating in the theory of positional differential games \cite{MR745789,MR1320507}, and notions of differentiability adapted to translations of the history, such as co-invariant derivatives \cite{MR1783365}. Within the co-invariant framework, \cite{MR4163474} studied an optimal control problem with constant delay and established the equivalence between minimax and viscosity solutions. 

For differential games with constant and distributed delays, \cite{MR4264642} introduced a notion of viscosity solution based on finite-dimensional projections obtained through constant extensions of histories, and characterized the value functional as the unique viscosity solution. Our construction builds on the projection argument of \cite{MR4264642}, but replaces constant extensions with polygonal ones, which are better suited to Lipschitz histories. The resulting notion of viscosity solution is consistent with those introduced in \cite{MR4163474,MR2729685}. 

A substantial literature is also devoted to optimality conditions for systems with prescribed delays in either the state or the control; see, for instance, \cite{MR231007,MR247556,MR477959,MR3702857,MR3705373}. These results typically rely on a suitable linearization of the state equation. For equations with state-dependent delays, \cite{MR1014944} established such a linearization within a Fr\'echet differentiability framework and showed that the resulting linearized equation is no longer itself a state-dependent delay equation. 

The same phenomenon occurs in our setting. Here, however, we derive the linearization with respect to localized perturbations of the initial history measured in the uniform norm. Finally, semiconcavity is a fundamental regularity property in finite-dimensional Hamilton--Jacobi theory and optimal control; see \cite{MR2041617}. To the best of our knowledge, the only available semiconcavity result for an optimal control problem involving a time delay concerns a minimum-time problem with constant delay \cite{MR4893227}. In the present work, we adapt classical semiconcavity arguments and combine them with more refined estimates based on Taylor expansions. This allows us to control the second-order effects arising from the state dependence of the delay.

\medskip

The remainder of the paper is organized as follows. \Cref{sec:setting} introduces the optimal control problem and the standing assumptions, and collects the basic properties of the associated trajectories and value functional. In \Cref{sec:dpp}, we establish the dynamic programming principle and characterize the value functional as a viscosity solution. \Cref{sec:unique_viscosity} contains the comparison principle and the ensuing uniqueness result. \Cref{sec:optimality} develops the linearization of the state equation and derives the Pontryagin minimum principle together with a generalized transversality condition. Finally, \Cref{sec:semiconcave} establishes semiconcavity estimates for the value functional.

\section{Setting of the problem}\label{sec:setting}
Let \(I=[t_0,T]\) be the time interval and set \(\mathbb{G}:=I\times \operatorname{Lip}([-\bar\tau,0];\mathbb R^n)\). We will use \(\|\cdot\|\) to denote the standard norm in \(\mathbb{R}^n\), and use the standard function spaces \(L^1\), \(L^2\), \(L^\infty\) and \(H^1\) on \([-\bar\tau,0]\) along with their norms \(\|\cdot\|_1\), \(\|\cdot\|_2\), \( \|\cdot\|_\infty\) and \(\|\cdot\|_{H^1}\) respectively. Let $\operatorname{Lip}(v)$ denote the Lipschitz constant of a Lipschitz function $v$.


Let \(A\subset \mathbb R^m\) be a nonempty, convex and compact set.
For \(t_{0}\leq t<s\leq T\), the set of admissible controls is
\[
    \mathcal A_{t,s}
    :=
    \left\{
        \alpha:(t,s)\to A:\ \alpha \text{ is measurable}
    \right\}.
\]
For \(R>0\), set
\[
    P(R)
    :=
    \left\{
        w\in \operatorname{Lip}([-\bar\tau,0];\mathbb R^n):
        \|w\|_\infty\le R
    \right\},
\]
and
\[
    V(R)
    :=
    \left\{
        w\in \operatorname{Lip}([-\bar\tau,0];\mathbb R^n):
        \|w\|_\infty\le R,\ \|\dot w\|_\infty\le R
    \right\}.
\]

Define
\begin{equation}\label{domain_g}
  \mathcal{X}= \left\{
            (z,w): w\in \operatorname{Lip}([-\bar\tau,0];\mathbb R^n), z=w(0)
        \right\}
\end{equation}
with norm $\left\Vert (z,w) \right\Vert _{\mathcal{X}}=\left\Vert z \right\Vert +\left\Vert w \right\Vert _{1}$.

\begin{assumption}\label{ass:standing}
We impose the following standing assumptions.

\begin{enumerate}
    \item The functions
    \[
        f:I\times \mathbb R^n\times \mathbb R^n\times A\to \mathbb R^n,
        \qquad
        L:I\times \mathbb R^n\times \mathbb R^n\times A\to \mathbb R
    \]
    are continuous.

    \item There exists \(\lambda_f>0\) such that
    \[
    \|f(t,z,y,a)-f(t,z',y',a)\|\leq \lambda_f \bigl(\|z-z'\|+\|y-y'\|\bigr)
    \]
    for all \(t\in I\), \(a\in A\) and \(z,y,z',y'\in \mathbb{R}^n\), and thus
    \[
     \|f(t,z,y,a)\|\leq   C_{f}\bigl(1+\|z\|+\|y\|\bigr)
    \]
    for some \(C_f>0\).
 \item For every \((t,z,y)\in I\times\mathbb R^n\times\mathbb R^n\), the set
\[
\mathcal Q(t,z,y)
:=
\left\{
    \bigl(f(t,z,y,a),\ell\bigr)
    \in\mathbb R^n\times\mathbb R:
    a\in A,\quad
    \ell\ge L(t,z,y,a)
\right\}
\]
is convex.
    \item For every \(R>0\), there exists
    \(\lambda_{L}(R)>0\) such that
    \begin{equation}\label{eq:L_lipschitz}
    \begin{aligned}
        |L(t,z,y,a)-L(t,z',y',a)|\le
        \lambda_{L}(R)\bigl(\|z-z'\|+\|y-y'\|\bigr)
    \end{aligned}
    \end{equation}
    for all \(t\in I\), \(a\in A\), and
    \(z,y,z',y'\in B_R:=\left\{
        x\in\mathbb{R}^n:\left\Vert  x \right\Vert\le R
    \right\}\).

    \item There exists \(C_{L}>0\) such that
    \[
        |L(t,z,y,a)|
        \le
        C_{L}\bigl(1+\|z\|+\|y\|\bigr)
    \]
    for all \((t,z,y,a)\in I\times\mathbb R^n\times\mathbb R^n\times A\).

    \item The terminal cost $g: \mathcal{X}\to \mathbb R$ is locally Lipschitz in the following sense. For every \(R>0\), there exists
    \(\lambda_g(R)>0\) such that
    \[
        |g(z_1,w_1)-g(z_2,w_2)|
        \le
        \lambda_g(R)
        \bigl(
            \left\Vert (z_{1},w_{1})-(z_{2},w_{2}) \right\Vert_{\mathcal{X} } 
        \bigr)
    \]
    for all $(z_{i},w_{i})\in \mathcal{X} $ with \(w_i\in P(R)\), \(i=1,2\).

    \item The delay function \(\tau:I\times\mathbb R^n\to [0,\infty)\) is \(C^1\) with supremum \(\bar\tau<\infty\). Moreover, there exists \(\lambda_\tau>0\) such that
    \[
        |\tau(t,z)-\tau(t,y)|
        \le
        \lambda_\tau\|z-y\|
        \qquad
        \forall\, t\in I,\quad z,y\in\mathbb R^n.
    \]
\end{enumerate}
\end{assumption}

For \((t,w)\in\mathbb G\) and \(\alpha\in\mathcal A_{t,T}\), we consider the
state-dependent delay equation
\begin{equation}\label{eq:FDE}
\begin{cases}
\dot x(s)
=
f\bigl(s,x(s),x_s(-\tau(s,x(s))),\alpha(s)\bigr),
& s\in(t,T),\\[2mm]
x(t+\theta)=w(\theta),
& \theta\in[-\bar\tau,0],
\end{cases}
\end{equation}
where
\[
    x_s(\theta):=x(s+\theta),
    \qquad
    \theta\in[-\bar\tau,0].
\]
For each
\((t,w)\in\mathbb G\) and \(\alpha\in\mathcal A_{t,T}\), there exists a unique solution, denoted by \(x(\cdot\,|\,t,w,\alpha)\), i.e., \(x:[t-\bar{\tau},T]\to \mathbb{R}^n\) is Lipschitz, satisfies the differential equation for a.e. \(s\in (t,T)\) with \(x_t=w\).

We now show the well-posedness of \eqref{eq:FDE}. Fix \((t,w)\in \mathbb G\) and \(\alpha\in \mathcal{A}_{t,T}\). Define
\[
    F(s,\phi)
    :=
    f\bigl(s,\phi(0),\phi(-\tau(s,\phi(0))),\alpha(s)\bigr),
    \qquad
    (s,\phi)\in [t,T]\times H^1(-\bar\tau,0;\mathbb R^n).
\]
Then \(s\mapsto F(s,\phi)\) is measurable for every fixed \(\phi \in H^1(-\bar\tau,0;\mathbb R^n)\). We will show that \(\phi\mapsto F(s,\phi)\) is continuous for every fixed \(s\in [t,T]\). Since the continuity of \(\phi\mapsto \phi(0)\) is guaranteed by the Sobolev embedding \(H^1(-\bar\tau,0;\mathbb R^n)\hookrightarrow C([-\bar\tau,0];\mathbb R^n)\), it suffices to show that
\[
\phi \mapsto \phi(-\tau(s,\phi(0)))
\]
is continuous. Let \(\phi,\psi\in H^1(-\bar\tau,0;\mathbb R^n)\). We estimate
\[
\begin{aligned}
&\|\psi(-\tau(s,\psi(0)))-\phi(-\tau(s,\phi(0)))\|\\
&\,\le \|\psi(-\tau(s,\psi(0)))-\psi(-\tau(s,\phi(0)))\|+\|\psi(-\tau(s,\phi(0)))-\phi(-\tau(s,\phi(0)))\|\\
&\,=:D_1+D_2.
\end{aligned}
\]
By H\"older's inequality, we have
\[
D_1\le \|\psi\|_{H^1}|\tau(s,\psi(0))-\tau(s,\phi(0))|^{1/2}.
\]
By the Sobolev embedding, we estimate
\[
D_2 \le 
\left\|
    \int_{-\tau(s,\phi(0))}^{0}
    \bigl(\dot\phi(\theta)-\dot\psi(\theta)\bigr)\,d\theta
\right\|+\|\psi(0)-\phi(0)\|\leq C \|\phi-\psi\|_{H^1}.
\]
Combining the above estimates gives the continuity of \(\phi\mapsto F(s,\phi)\).

Next we show a Lipschitz property of \(\phi\mapsto F(s,\phi)\). If \(\phi,\psi\in H^1(-\bar\tau,0;\mathbb R^n)\) satisfy
\[
    \|\dot\phi\|_{L^\infty},
    \|\dot\psi\|_{L^\infty}
    \le r
\]
for some \(r>0\), then, by the global Lipschitz continuity of \(f\),
\[
\begin{aligned}
    \|F(s,\phi)-F(s,\psi)\|
    &\le
    \lambda_f\left(\|\phi(0)-\psi(0)\|+D_1+D_2\right).
\end{aligned}
\]
In this case, by the Lipschitz continuity of \(\tau\), we have
\[
D_1\leq r|\tau(s,\psi(0))-\tau(s,\phi(0))|\le C(r)\|\psi-\phi\|_{H^1}.
\]
Consequently, there exists \(C(r)>0\) such that
\[
 \|F(s,\phi)-F(s,\psi)\|\le C(r)\|\psi-\phi\|_{H^1}
\]
for all \(s\in [t,T]\).

With a slight modification of the contraction argument used in the proof \cite[Theorem 1.1]{MR4785300}, the above estimates yield the local existence and uniqueness of the solution. Using the sublinear growth condition of \(f\) and a standard continuation argument, we then obtain a unique solution \(x:[t-\bar{\tau},T]\to \mathbb{R}^n\). 

We impose the standing monotonicity condition on the delay function.

\begin{assumption}\label{ass:monotone-delay}
For every \(r>0\), there exists \(d(r)>0\) such that, for every solution
\(x(\cdot)\) of \eqref{eq:FDE} with \(t\in I\), \(w\in P(r)\) and
\(\alpha\in\mathcal A_{t,T}\), we have
\[
    1-\tau_t(s,x(s))
      -\tau_z(s,x(s))\cdot \dot x(s)
    \ge d(r)
\]
for a.e. \(s\in[t,T]\).
\end{assumption}

Equivalently, along every admissible trajectory with initial
history \(w\in P(r)\), the map
\[
    s\mapsto s-\tau(s,x(s))
\]
is strictly increasing, with derivative bounded below by \(d(r)>0\).
This property will be used repeatedly when performing changes of variables
in delayed terms.

For \((t,w)\in \mathbb{G}\) and \(\alpha\in \mathcal{A}_{t,T}\), we define the cost functional 
\[
    J(t,w,\alpha)
    =
    \int_t^T
    L\bigl(s,x(s),x(s-\tau(s,x(s))),\alpha(s)\bigr)\,ds
    +
    g(x(T),x_T(\cdot)),
\]
where \(x(\cdot)=x(\cdot|\, t,w,\alpha)\). The corresponding value functional is
\begin{equation}\label{eq:value}
    \rho(t,w)
    :=
    \inf_{\alpha\in\mathcal A_{t,T}}J(t,w,\alpha).
\end{equation}
We say \(\alpha\in \mathcal A_{t,T}\) is an optimal control for \((t,w)\) if the above infimum is attained at \(\alpha\). In this case, letting \(x(\cdot)=x(\cdot|\, t,w,\alpha)\), we call \((x,\alpha)\) an optimal pair. 

\section{Dynamic programming principle and viscosity solutions}\label{sec:dpp}

In this section, we first present the dynamic programming principle for the value functional~\eqref{eq:value}. Then we introduce the notion of viscosity solution of a Hamilton--Jacobi equation with co-invariant derivatives and finally prove that the value functional is a viscosity solution.

\begin{theorem}[Dynamic programming principle]\label{thm:DPP}
Let \((t,w)\in\mathbb G\), \(t<T\), and let \(h>0\) be such that
\(t+h\le T\). Then
\[
\begin{aligned}
    \rho(t,w)
    =
    \inf_{\alpha\in\mathcal A_{t,t+h}}
    \bigg\{
        \int_t^{t+h}
        L\bigl(s,x(s),x(s-\tau(s,x(s))),\alpha(s)\bigr)\,ds
        +
        \rho(t+h,x_{t+h})
    \bigg\},
\end{aligned}
\]
where \(x=x(\cdot\,|\,t,w,\alpha)\). Moreover, if
\(\alpha^*\) is an optimal control for \((t,w)\), then the above infimum is
attained by \(\alpha^*\).
\end{theorem}

\begin{proof}
Fix \(\alpha^1\in \mathcal{A}_{t,t+h}\), and let \(x^1\) be the corresponding
solution on \([t,t+h]\) with initial history \(w\). By the definition of
\(\rho\), for every \(\varepsilon>0\) there exists
\(\alpha^2\in \mathcal{A}_{t+h,T}\) such that, if \(x^2\) is the solution on
\([t+h,T]\) with initial history \(x^1_{t+h}\), then
\[
\begin{aligned}
    \rho(t+h,x^1_{t+h})+\varepsilon
    \ge
    \int_{t+h}^T
    L\bigl(s,x^2(s),x^2(s-\tau(s,x^2(s))),\alpha^2(s)\bigr)\,ds
    +
    g(x^2(T),x^2_T).
\end{aligned}
\]
Define the concatenated control
\[
    \alpha^3(s)
    :=
    \begin{cases}
        \alpha^1(s), & s\in(t,t+h],\\
        \alpha^2(s), & s\in(t+h,T).
    \end{cases}
\]
By the uniqueness of the solution to \eqref{eq:FDE}, the trajectory associated
with \(\alpha^3\in \mathcal{A}_{t,T}\) coincides with \(x^1\) on \([t,t+h]\) and with \(x^2\)
on \([t+h,T]\). Therefore
\[
\begin{aligned}
    \rho(t,w)
    &\le J(t,w,\alpha^3) \\
    &\le
    \int_t^{t+h}
    L\bigl(s,x^1(s),x^1(s-\tau(s,x^1(s))),\alpha^1(s)\bigr)\,ds
    +
    \rho(t+h,x^1_{t+h})
    +
    \varepsilon.
\end{aligned}
\]
Since \(\alpha^1\) and \(\varepsilon\) are arbitrary, we deduce
\[
\begin{aligned}
    \rho(t,w)
    \le
    \inf_{\alpha\in\mathcal A_{t,t+h}}
    \bigg\{
        \int_t^{t+h}
        L\bigl(s,x(s),x(s-\tau(s,x(s))),\alpha(s)\bigr)\,ds
        +
        \rho(t+h,x_{t+h})
    \bigg\}.
\end{aligned}
\]

Conversely, choose \(\alpha^4\in\mathcal A_{t,T}\) such that
\begin{equation}\label{eq:alpha_4}
    \rho(t,w)+\varepsilon
    \ge
    J(t,w,\alpha^4),    
\end{equation}
and let \(x^4=x(\cdot\,|\,t,w,\alpha^4)\). By the definition of \(\rho\), we have
\[
\begin{aligned}
    \rho(t+h,x^4_{t+h})
    \le& 
    J(t+h,x^4_{t+h},\alpha^4|_{(t+h,T)})\\
    \le&
    J(t,w,\alpha^4)-\int_{t}^{t+h}
    L\bigl(s,x^4(s),x^4(s-\tau(s,x^4(s))),\alpha^4(s)\bigr)\,ds.
\end{aligned}
\]
Hence
\[
\begin{aligned}
    \rho(t,w)+\varepsilon
    \ge
    \int_t^{t+h}
    L\bigl(s,x^4(s),x^4(s-\tau(s,x^4(s))),\alpha^4(s)\bigr)\,ds
    +
    \rho(t+h,x^4_{t+h}).
\end{aligned}
\]
Letting
\(\varepsilon\downarrow0\) yields the reverse inequality. 

In particular, if \(\alpha^*\) is optimal, one can select \(\alpha^4=\alpha^*\) with \(\varepsilon=0\) in \eqref{eq:alpha_4} and then obtain
\[
\rho(t,w)
    \ge
    \int_t^{t+h}
    L\bigl(s,x^*(s),x^*(s-\tau(s,x^*(s))),\alpha^*(s)\bigr)\,ds
    +
    \rho(t+h,x^*_{t+h}),
\]
where \(x^*\) is the optimal trajectory associated with \(\alpha^*\). Thus the infimum in the dynamic programming formula is attained at \(\alpha^*\).
\end{proof}
We next prove several properties for solutions of \eqref{eq:FDE}.

\begin{lemma}\label{lem:trajectory_bounds}

For every \(r>0\), there exist constants \(R=R(r)>0\) and
\(\lambda=\lambda(r)>0\) such that, for any \(t\in[t_0,T]\),
\(w\in P(r)\) with $z=w(0)$, and \(\alpha\in\mathcal A_{t,T}\), the solution
\(x=x(\cdot\,|\,t,w,\alpha)\) satisfies
\[
    \|x(t_1)-x(t_2)\|
    \le
    \lambda |t_1-t_2|,
    \qquad
    t_1,t_2\in[t,T],
\]
and
\[
    x_s\in P(R),
    \qquad
    s\in[t,T].
\]
\end{lemma}

\begin{proof}
Denote $D(\xi):=\xi-\tau(\xi,x(\xi))$ for $\xi \in [t,T]$. For \(s\in[t,T]\), we have
\[
    x(s)
    =
    z+
    \int_t^s
    f\bigl(\xi,x(\xi),x(D(\xi)),\alpha(\xi)\bigr)\,d\xi.
\]
Using the growth condition on \(f\), we obtain
\[
    \|x(s)\|
    \le
    \|z\|
    +
    C_{f}\int_t^s 1+\|x(\xi)\|+\|x(D(\xi))\|\,d\xi
\]
Set
\[
    k(s):=\max_{\xi\in[t-\bar\tau,s]}\|x(\xi)\|, \quad s\in [t,T].
\]
We derive
\[
    k(s)
    \le
    \left\Vert w \right\Vert _{\infty}+C_{f}\int_t^s (1+2k(\xi))\,d\xi,\quad s\in [t,T],
\]
since $D(\xi)\in [t-\bar{\tau},\xi]$ for all $\xi \in [t,s]$. Then Gronwall's inequality gives
\[
\left\Vert x(s) \right\Vert \leq k(s) \leq\left(\|w\|_{\infty}+\frac{1}{2}\right) e^{2 C_f(s-t)}-\frac{1}{2}
\]
for all $s \in [t,T]$. Together with the bound on the initial history $w$, this implies
\(x_s\in P(R)\) for some \(R=R(r)\). Also
\begin{equation}\label{eq:C_f}
    \|f(s,x(s),x(D(s)),\alpha(s))\|
    \le 
    C_{f}(1+2k(s))
    \le
    C_{f}(1+2\left\Vert w \right\Vert_{\infty})e^{2C_{f}(s-t)}.
\end{equation}
Therefore, there exists $\lambda (r)>0$ such that
\[
    \|x(t_1)-x(t_2)\|
    \le
    \lambda(r)|t_1-t_2|,
    \qquad
    t_1,t_2\in[t,T].
\]
\end{proof}

\begin{remark}
For fixed $(t,w)\in \mathbb{G}$, we deduce from \eqref{eq:C_f} that there exists $\delta >0$ such that the solution satisfies
\[
  \left\Vert \dot{x} (s) \right\Vert \leq C, \quad \text{a.e. } s\in [t,t+\delta ]
\]
for all $C>C_{f}(1+2\left\Vert w \right\Vert_{\infty})$ and all $\alpha$.
\end{remark}

\begin{lemma}\label{lem:history_shift}
For every $R>0$, there exists \(C(R)>0\) such that, for all \((t,w)\in\mathbb G\) with \(w\in V(R)\), all \(\alpha\in \mathcal{A}_{t,T}\), and all $s\in (t,T]$, the corresponding solution satisfies
\[
    \|x_{s}-w\|_\infty
    \le
    C(R)|s-t|.
\]
\end{lemma}

\begin{proof}
For \(\xi\in[-\bar\tau,t-s]\), we have
\[
    \|x_{s}(\xi)-w(\xi)\|
    =
    \|w(s-t+\xi)-w(\xi)\|
    \le
    C(R)|s-t|.
\]
If \(\xi\in(t-s,0]\), by \Cref{lem:trajectory_bounds}, we have
\[
\begin{aligned}
    \|x_{s}(\xi)-w(\xi)\|
    &\le
    \|x(s+\xi)-x(t)\|+\|w(0)-w(\xi)\| \\
    &\le
    C(R)(|s+\xi-t|
    +
    |\xi|) \\
    &\le
    C(R)|s-t|.
\end{aligned}
\]
\end{proof}

\begin{lemma}\label{lem:continuous_dependence}
For every \(R>0\), there exists \(\lambda_*(R)>0\) such that, for all
\(t\in I\), all \(w_i\in V(R)\) with \(z_{i}=w_{i}(0)\), and all \(\alpha\in\mathcal A_{t,T}\),  \(i=1,2\), the corresponding solutions \(x_i=x(\cdot; t ,w_i,\alpha)\) satisfy
\[
\begin{aligned}
    &\sup_{\xi\in [t,T]}\|x_1(\xi)-x_2(\xi)\|
    +
    \|(x_1)_T-(x_2)_T\|_1 \\
    &\quad
    +
    \left|
        \int_t^T
        \Bigl[
        L\bigl(s,x_1(s),x_1(D_1(s)),\alpha(s)\bigr)
        -
        L\bigl(s,x_2(s),x_2(D_2(s)),\alpha(s)\bigr)
        \Bigr]\,ds
    \right| \\
    &\le
    \lambda_*(R)
    \bigl(
        \|z_1-z_2\|+\|w_1-w_2\|_1
    \bigr),
\end{aligned}
\]
where \(D_i(s):=s-\tau(s,x_i(s))\), \(i=1,2\).
\end{lemma}

\begin{proof}
By the Lipschitz continuity of \(f\), for every \(\xi\in[t,T]\),
\[
\begin{aligned}
    &\|x_1(\xi)-x_2(\xi)\|\\
    &\le
    \|z_1-z_2\|+
    C(R)\int_t^\xi
    \Bigl(
        \|x_1(s)-x_2(s)\|
        +
        \|x_1(D_1(s))
          -x_2(D_2(s))\|
    \Bigr)\,ds.
\end{aligned}
\]
We estimate the delayed term by writing
\[
\begin{aligned}
&\|x_1(D_1(s))
          -x_2(D_2(s))\| \\
&\le
\|x_1(D_1(s))
          -x_2(D_1(s))\|+
\|x_2(D_1(s))
          -x_2(D_2(s))\|.
\end{aligned}
\]
Using \Cref{ass:monotone-delay} for the change of variables
\(\eta=D_1(s)\), the Lipschitz continuity of \(\tau\), and the
uniform Lipschitz bound on \(x_2\) by \Cref{lem:trajectory_bounds}, we obtain
\[
\begin{aligned}
&\int_t^\xi
\|x_1(D_1(s))
          -x_2(D_2(s))\|\,ds \\
&\le
C(R)\left(\|w_1-w_2\|_1
+\int_t^\xi \|x_1(s)-x_2(s)\|\,ds\right).
\end{aligned}
\]
Consequently,
\[
    \|x_1(\xi)-x_2(\xi)\|
    \le 
    C(R)\left(\|z_1-z_2\|+\|w_1-w_2\|_1+\int_t^\xi \|x_1(s)-x_2(s)\|\,ds\right)
\]
Then Gronwall's inequality yields
\begin{equation}\label{eq:continuous_dependence_state}
    \|x_1(\xi)-x_2(\xi)\|
    \le
    C(R)\bigl(\|z_1-z_2\|+\|w_1-w_2\|_1\bigr),
    \qquad
    \xi\in[t,T].
\end{equation}
The estimate for the difference of the running costs follows from the
local Lipschitz continuity of \(L\) and the same estimate for the delayed term.
Finally,
\[
    \|(x_1)_s-(x_2)_s\|_1
    \le
    \|w_1-w_2\|_1
    +
    (T-t_0)
    \max_{\xi\in[t,s]}\|x_1(\xi)-x_2(\xi)\|,
\]
and the desired estimate follows from
\eqref{eq:continuous_dependence_state}.
\end{proof}

\subsection{Viscosity solutions}

Here, we formulate a Hamilton--Jacobi equation and its notion of viscosity solution.
\begin{definition}[Co-invariant differentiability]\label{Co-invariant differentiability}
Let \(\varphi:\mathbb G\to\mathbb R\). We say that \(\varphi\) is
co-invariantly differentiable at \((t,w)\in \mathbb{G}\), if there exist
\[
    \partial^{ci}_t\varphi(t,w)\in\mathbb R,
    \qquad
    \nabla^{ci}\varphi(t,w)\in\mathbb R^n,
\]
such that, for every \(s\in(t,T]\) and every
\(y\in\Lambda(t,w)\),
\[
    \varphi(s,y_s)-\varphi(t,y_t)
    =
    \partial^{ci}_t\varphi(t,w)(s-t)
    +
    \nabla^{ci}\varphi(t,w)\cdot(y(s)-y(t))
    +
    o(s-t),
\]
where
\[
    \Lambda(t,w)
    :=
    \left\{
        y\in \operatorname{Lip}([t-\bar\tau,\infty);\mathbb R^n):
        y(t+\xi)=w(\xi)
        \text{ for } \xi\in[-\bar\tau,0]
    \right\},
\]
and \(o(\cdot)\) can depend on \(y\). Here \(y_s(\theta)=y(s+\theta)\), \(\theta\in [-\bar \tau, 0]\) for $s\in [t,T]$. 
\end{definition}
\noindent For more details on $ci$-calculus we refer to \cite{MR1783365}. 

Define the Hamiltonian
\[
    H(t,z,y,p)
    =
    \min_{a\in A}
    \bigl\{
        f(t,z,y,a)\cdot p+L(t,z,y,a)
    \bigr\}
\]
for \((t,z,y,p)\in I\times\mathbb R^n\times\mathbb R^n\times\mathbb R^n\). For \(\varphi:\mathbb{G}\to \mathbb{R}\), the Hamilton--Jacobi equation associated with the control problem \eqref{eq:FDE} is
\begin{equation}\label{eq:HJE}
\partial^{ci}_t\varphi(t,w)
+
H\bigl(t,z,w(-\tau(t,z)),\nabla^{ci}\varphi(t,w)\bigr)
=0,\quad t<T,\ z=w(0),
\end{equation}
with the terminal condition
\begin{equation}\label{eq:terminal}
    \varphi(T,w)=g(z,w), \quad w\in  \operatorname{Lip}([-\bar\tau,0];\mathbb R^n).
\end{equation}

Fix any \((t,w)\in\mathbb G\) with \(z=w(0)\). For \(s\in(t,T]\) and
\(x\in\mathbb R^n\), define the polygonal extension \(v\) connecting \((t,z)\) and \((s,x)\) by
\[
    v(\xi)
    :=
    \begin{cases}
        w(\xi-t), & \xi\in[t-\bar\tau,t),\\[1mm]
        z+\dfrac{x-z}{s-t}(\xi-t), & \xi\in[t,s],\\[3mm]
        x, & \xi>s.
    \end{cases}
\]
We denote this singleton by
\[
    \Lambda_0(t,w;s,x):=\{v\},
\]
which is a subset of \(\Lambda(t,w)\).

For \(\varphi:\mathbb G\to\mathbb R\), define the finite-dimensional
projection with base point $(t,w)$ by
\begin{equation}\label{eq:bar_phi}
    \bar\varphi(s,x)
    :=
    \begin{cases}
        \varphi(s,v_s), & (s,x)\in(t,T]\times\mathbb R^n,\\
        \varphi(t,w), & (s,x)=(t,z),
    \end{cases}
\end{equation}
where \(v_s(\theta):=v(s+\theta)\), \(\theta\in[-\bar\tau,0]\), denotes the history of \(v\) at time \(s\).

Let \(\Phi\) be the class of functionals
\(\varphi:\mathbb G\to\mathbb R\) satisfying the following local
Lipschitz condition: for every \(R>0\), there exists
\(\lambda_\varphi(R)>0\) such that
\[
    |\varphi(t_1,w_1)-\varphi(t_2,w_2)|
    \le
    \lambda_\varphi(R)
    \bigl(
        |t_1-t_2|
        +
        \|z_1-z_2\|
        +
        \|w_1-w_2\|_1
    \bigr)
\]
for any  \(t_i\in I\), \(w_i \in V(R)\) with \(z_{i}=w_{i}(0)\), \(i=1,2\).
For \((t,z)\in I\times\mathbb R^n\), $a<b$, and \(C>0\), define the cone
\begin{equation}\label{kone}
    K(a,b,z,C)
    :=
    \left\{
        (s,x)\in [a,b]\times\mathbb R^n:
        \|x-z\|\le C(s-a)
    \right\}.
\end{equation}
\begin{definition}[Viscosity solution]\label{def:viscosity}
A functional \(\varphi\in\Phi\) is a viscosity supersolution of
\eqref{eq:HJE} if for any \((t,w)\in\mathbb G\) with \(z=w(0)\), \(t<T\), \(\phi\in C^1(\mathbb R\times\mathbb R^n)\), $C>C_{f}(1+2\left\Vert w \right\Vert_{\infty})$, and \(\delta>0\),
\[
    \bar\varphi(t,z)-\phi(t,z)
    \le
    \bar\varphi(s,x)-\phi(s,x)
    \qquad
    \forall\,(s,x)\in     K(t,t+\delta,z,C)
\]
implies
\[
    \partial_t\phi(t,z)
    +
    H\bigl(t,z,w(-\tau(t,z)),\nabla\phi(t,z)\bigr)
    \le 0.
\]

Similarly, \(\varphi\in\Phi\) is a viscosity subsolution if \(\varphi\) satisfies
\[
    \bar\varphi(t,z)-\phi(t,z)
    \ge
    \bar\varphi(s,x)-\phi(s,x)
    \qquad
    \forall\,(s,x)\in K(t, t+\delta ,z,C)
\]
implies
\[
    \partial_t\phi(t,z)
    +
    H\bigl(t,z,w(-\tau(t,z)),\nabla\phi(t,z)\bigr)
    \ge 0.
\]
If a functional \(\varphi\) is both a viscosity subsolution and a viscosity
supersolution, and satisfies \eqref{eq:terminal}, then we say \(\varphi\) is a viscosity solution of \eqref{eq:HJE} with the terminal condition \eqref{eq:terminal}.
\end{definition}

\begin{remark}
If \(\rho\) is co-invariantly differentiable at \((t,w)\), then it can be checked that \(\rho\) satisfies \eqref{eq:HJE} at that point in the classical sense by the dynamic programming principle. 
\end{remark}

Let \((t,w)\in \mathbb{G}\) with \(z=w(0)\) and \(\varphi\in\Phi\). For \(\ell>0\) and \(\theta \in\mathbb R^n\), let
\(v\in\Lambda_0(t,w;t+h\ell,z+h\theta)\). Define the upper and lower Dini
derivatives by
\[
    \partial^+\varphi(t,w)(\ell,\theta)
    :=
    \limsup_{h\downarrow0}
    \frac{
        \varphi(t+h\ell,v_{t+h\ell})-\varphi(t,w)
    }{h},
\]
and
\[
    \partial^-\varphi(t,w)(\ell,\theta)
    :=
    \liminf_{h\downarrow0}
    \frac{
        \varphi(t+h\ell,v_{t+h\ell})-\varphi(t,w)
    }{h}.
\]

The delay superdifferential and subdifferential are closed and convex sets defined respectively by
\[
\begin{aligned}
    \partial^{\mathrm{delay}}_+\varphi(t,w)
    :=
    \bigl\{
        (q,p)\in\mathbb R\times\mathbb R^n:
        \partial^+\varphi(t,w)(\ell,\theta)
        \le
        q\ell+p\cdot \theta \\
        \hfill
        \forall\,(\ell,\theta)\in (0,\infty)\times\mathbb R^n
    \bigr\},
\end{aligned}
\]
and
\[
\begin{aligned}
    \partial^{\mathrm{delay}}_-\varphi(t,w)
    :=
    \bigl\{
        (q,p)\in\mathbb R\times\mathbb R^n:
        \partial^-\varphi(t,w)(\ell,\theta)
        \ge
        q\ell+p\cdot \theta \\
        \hfill
        \forall\,(\ell,\theta)\in (0,\infty)\times\mathbb R^n
    \bigr\}.
\end{aligned}
\]

\begin{lemma}\label{lem:value_lipschitz}
The value functional \(\rho\) belongs to \(\Phi\).
\end{lemma}

\begin{proof}
Fix \(R>0\), \(t\in I\), and \(w,\hat w\in V(R)\) with \(z=w(0)\) and \(\hat{z} =\hat{w} (0)\). For \(\varepsilon>0\), choose
\(\hat\alpha\in\mathcal A_{t,T}\) such that
\[
    \rho(t,\hat w)+\varepsilon
    \ge
    J(t,\hat w,\hat\alpha).
\]
Let \(x\) and \(\hat x\) be the solutions corresponding to
\((t,w,\hat\alpha)\) and \((t,\hat w,\hat\alpha)\) respectively. By the definition of
\(\rho\),
\[
\begin{aligned}
    \rho(t,w)-\rho(t,\hat w)
    &\le
    \int_t^T
    \Bigl[
    L\bigl(s,x(s),x(s-\tau(s,x(s))),\hat\alpha(s)\bigr) \\
    &\qquad\qquad
    -
    L\bigl(s,\hat x(s),\hat x(s-\tau(s,\hat x(s))),\hat\alpha(s)\bigr)
    \Bigr]\,ds \\
    &\quad
    +
    g(x(T),x_T)-g(\hat x(T),\hat x_T)
    +
    \varepsilon.
\end{aligned}
\]
Using \Cref{lem:continuous_dependence} and the local Lipschitz
continuity of \(g\), we obtain
\[
    \rho(t,w)-\rho(t,\hat w)
    \le
    C(R)\bigl(\|z-\hat z\|+\|w-\hat w\|_1\bigr)
    +
    \varepsilon.
\]
Interchanging \(w\) and \(\hat w\), and then letting
\(\varepsilon\downarrow0\), we derive
\[
    |\rho(t,w)-\rho(t,\hat w)|
    \le
    C(R)\bigl(\|z-\hat z\|+\|w-\hat w\|_1\bigr).
\]

It remains to prove the Lipschitz continuity in time. Let
\(t_1<t_2\). By the dynamic programming principle,
for every \(\varepsilon>0\) there exists \(\alpha\in\mathcal A_{t_1,T}\) such that
the corresponding trajectory \(x=x(\cdot\,|\,t_1,w,\alpha)\) satisfies
\[
\begin{aligned}
    \rho(t_1,w)
    &\le
    \int_{t_1}^{t_2}
    L\bigl(s,x(s),x(s-\tau(s,x(s))),\alpha(s)\bigr)\,ds
    +
    \rho(t_2,x_{t_2})
    \le  \rho(t_1,w)+\varepsilon.
\end{aligned}
\]
Using the previous Lipschitz estimate in the history variable and \Cref{lem:history_shift}, we have
\[
    |\rho(t_2,x_{t_2})-\rho(t_2,w)|
    \le
    C(R)|t_2-t_1|.
\]
Moreover, by \Cref{lem:trajectory_bounds} and the growth assumption on
\(L\),
\[
    \int_{t_1}^{t_2}
    \bigl|
        L\bigl(s,x(s),x(s-\tau(s,x(s))),\alpha(s)\bigr)
    \bigr|\,ds
    \le
    C(R)|t_2-t_1|.
\]
Therefore
\[
  |\rho (t_{2},w(\cdot ))-\rho (t_{1},w(\cdot))|<C|t_1-t_2|+\varepsilon .
\]
Letting \(\varepsilon\downarrow 0\), we obtain
\[
    |\rho(t_1,w)-\rho(t_2,w)|
    \le
    C(R)|t_2-t_1|.
\]
Hence \(\rho\in\Phi\).
\end{proof}

\begin{theorem}\label{thm:value_viscosity}
The value functional \(\rho\) is a viscosity solution of
\eqref{eq:HJE}.
\end{theorem}

\begin{proof}
The terminal condition follows immediately from the definition of the
value functional. We prove the supersolution property; the subsolution
property is analogous.

Suppose that \(\rho\) is not a viscosity supersolution. Then there exist \((t_*,w_*)\in\mathbb G\) with \(z_{*}=w_{*}(0)\), \(t_*<T\), \(C_{*}>C_{f}(1+2\left\Vert w_{*}\right\Vert_{\infty})\) and a test function
\(\psi\in C^1(\mathbb R\times\mathbb R^n)\), such that
\[
    \bar\rho(t_*,z_*)-\psi(t_*,z_*)
    \le
    \bar\rho(s,x)-\psi(s,x)
\]
for all \((s,x)\in K(t_{*},t_{*}+\delta _{0},z_{*},C_{*})\), while
\[
    \partial_t\psi(t_*,z_*)
    +
    H\bigl(t_*,z_*,w_*(-\tau(t_*,z_*)),\nabla_{z}\psi(t_*,z_*)\bigr)
    >0.
\]
By the definition of \(H\), and by continuity, we may choose
\(\delta_0>0\) smaller and find \(\theta>0\) such that
\begin{equation}\label{eq:strict_ineq_test}
    \partial_t\psi(s,x)
    +
    \nabla_{z}\psi(s,x)\cdot
    f\bigl(s,x,v^{s,x}_s(-\tau(s,x)),a\bigr)+
    L\bigl(s,x,v^{s,x}_s(-\tau(s,x)),a\bigr)
    \ge
    \theta
\end{equation}
for all \((s,x)\in K(t_{*},t_{*}+\delta _{0},z_{*},C_{*})\) and \(a\in A\), where \(v^{s,x}\in \Lambda_0(t_*,w_*;s,x)\).

By \Cref{lem:trajectory_bounds}, every trajectory starting from \((t_*,w_*)\) satisfies
\[
    \|x(s)-z_*\|
    \le
    C_{*} |s-t_*|
\]
for \(s \in [t,t_{*}+h]\) for small $h<\delta _{0}$. Then \((s,x(s))\in K(t_{*},t_{*}+\delta _{0},z_{*},C_{*})\) for
\(s\in[t_*,t_*+h]\).

By \Cref{thm:DPP}, there exists
\(\alpha\in\mathcal A_{t_*,T}\) such that
\begin{equation}\label{eq:DPP_subsolution}
\begin{aligned}
    \rho(t_*,w_*)
    \ge
    \int_{t_*}^{t_*+h}
    L\bigl(s,x(s),x(s-\tau(s,x(s))),\alpha(s)\bigr)\,ds
    +
    \rho(t_*+h,x_{t_*+h})
    -
    \frac{\theta h}{2},
\end{aligned}
\end{equation}
where \(x=x(\cdot;t_*,w_*,\alpha)\). Denote \(v^{[s]}=v^{s,x(s)}\) for \(s\in (t_*,t_*+h]\). Since \(x\) is
Lipschitz, we have
\begin{equation}\label{eq:x-v}
\|x(s)-v^{[s]}(s)\|\le Ch
\end{equation}
for all \(s\in (t_*,t_*+h]\). Thus
\[
\begin{aligned}
    \|x_{t_*+h}-v^{[t_{*}+h]}_{t_*+h}\|_1
    &\le Ch^2.
\end{aligned}
\]
Using the Lipschitz continuity of \(\rho\), \eqref{eq:DPP_subsolution} gives
\begin{equation}\label{eq:DPP_polygon}
\begin{aligned}
    \rho(t_*+h,v^{[t_{*}+h]}_{t_*+h})-\rho(t_*,w_*)
    \le
    -
    \int_{t_*}^{t_*+h}
    L\bigl(s,x(s),x(s-\tau(s,x(s))),\alpha(s)\bigr)\,ds
    +
    \frac{\theta h}{2}
    +
    Ch^2 .
\end{aligned}
\end{equation}
On the other hand, the touching condition implies
\[
    \rho(t_*+h,v^{[t_{*}+h]}_{t_*+h})-\rho(t_*,w_*)
    \ge
    \psi(t_*+h,x(t_*+h))-\psi(t_*,z_*).
\]
Since \(x\) solves \eqref{eq:FDE},
\[
\begin{aligned}
    \psi(t_*+h,x(t_*+h))-\psi(t_*,z_*)
    =
    \int_{t_*}^{t_*+h}
    \Bigl[
        \partial_t\psi(s,x(s))
        +
        \nabla_z\psi(s,x(s))\cdot\dot x(s)
    \Bigr]\,ds.
\end{aligned}
\]
Combining this with \eqref{eq:DPP_polygon}, we obtain
\[
\begin{aligned}
    \frac{\theta h}{2}+Ch^2
    \ge
    \int_{t_*}^{t_*+h}
    \Bigl[
        \partial_t\psi(s,x(s))
        +
        \nabla_z\psi(s,x(s))\cdot
        f(s,x(s),x(s-\tau(s,x(s))),\alpha(s)) \\
        \qquad\qquad
        +
        L(s,x(s),x(s-\tau(s,x(s))),\alpha(s))
    \Bigr]\,ds.
\end{aligned}
\]
Again, \eqref{eq:x-v} implies that
\[
\begin{aligned}
    \frac{\theta h}{2}+Ch^2
    \ge
    \int_{t_*}^{t_*+h}
    \Bigl[
        \partial_t\psi(s,x(s))
        +
        \nabla_z\psi(s,x(s))\cdot
        f(s,x(s),v^{[s]}(s-\tau(s,x(s))),\alpha(s)) \\
        \qquad\qquad
        +
        L(s,x(s),v^{[s]}(s-\tau(s,x(s))),\alpha(s))
    \Bigr]\,ds.
\end{aligned}
\]
Therefore, by \eqref{eq:strict_ineq_test},
\[
    \frac{\theta h}{2}+Ch^2
    \ge
    \theta h.
\]
This is impossible for \(h>0\) sufficiently small. Hence \(\rho\) is a
viscosity subsolution.

The proof of the subsolution property is obtained by reversing the
touching inequality and using an \(\varepsilon\)-optimal control in the
opposite direction, with using a constant control $a^{*}\in A$ instead arbitrary control $a$ in \eqref{eq:strict_ineq_test}. Thus \(\rho\) is a viscosity solution of
\eqref{eq:HJE} with the terminal condition \eqref{eq:terminal}.
\end{proof}

\section{Uniqueness of viscosity solutions}\label{sec:unique_viscosity}

We shall prove the uniqueness of the viscosity solution. The proof strategy is from \cite[Lemma 7.6]{MR4264642}, adapted to polygonal extensions.

\begin{lemma}\label{lem:H_lipschitz_p}
For every \(t\in I\) and \(z,y,p,p'\in\mathbb R^n\), we have
\[
    |H(t,z,y,p)-H(t,z,y,p')|
    \le
    C_{f}(1+\|z\|+\|y\|)\|p-p'\|.
\]
\end{lemma}

\begin{proof}
Let \(a'\in A\) be a minimizer for \(H(t,z,y,p')\). Then
\[
\begin{aligned}
    H(t,z,y,p)-H(t,z,y,p')
    &\le
    f(t,z,y,a')\cdot(p-p')  \\
    &\le
    C_{f}(1+\|z\|+\|y\|)\|p-p'\|.
\end{aligned}
\]
Exchanging the roles of \(p\) and \(p'\) gives the reverse estimate.
\end{proof}

Let \((t_{*},w_{*})\in\mathbb G\) and $t_{*}<T_{*}\leq T$ with \(z_{*}=w_{*}(0)\), and let \(C_*>0\). Set $K:=K(t_{*},T_{*},z_{*},C_*)$, the cone defined in \eqref{kone}.
\begin{lemma}\label{lem:continuous_polygon_extension}
Let \(u^{r,y}\in \Lambda_0(t_{*},w_{*};r,y)\) and \(v^{s,x}\in \Lambda_0(t_{*},w_{*};s,x)\) for \(r,s>t\) and \((r,y),(s,x)\in K\). Then, there exists \(C=C(\operatorname{Lip}(w_{*}),C_*)\) such that
\[
    \|v^{s,x}_s-w_{*}\|_1
    \le C(s-t)
\]
and
\begin{equation}\label{second_estimate}
    \|u^{r,y}_r-v^{s,x}_s\|_1
    \le C(|r-s|+\|y-x\|).
\end{equation}
Moreover, there exists \(C=C(\left\Vert z_{*} \right\Vert ,\operatorname{Lip}(w_{*}),C_*)>0\) such that
\begin{equation}\label{final_estimate}
\left\|
u ^{r,y}_r\bigl(-\tau(r,y)\bigr)
-
v ^{s,x}_s\bigl(-\tau(s,x)\bigr)
\right\| \le
C\bigl(|r-s|+\|y-x\|\bigr).
\end{equation}
\end{lemma}

\begin{proof}
It suffices to prove the estimates in the case \( r,s < t_{*}+\bar{\tau}, \) since the other cases follow by the same argument with fewer contributions from the initial history.

We first estimate \(\|v^{s,x}_s-w_{*}\|_1\). To do so, we write
\[
\begin{aligned}
    \|v^{s,x}_s-w_{*}\|_1=
    \int_{-\bar\tau}^0 \|v^{s,x}(s+\xi)-w_{*}(\xi)\|\,d\xi =:\int_{-\bar\tau}^0 \sigma(\xi)\,d\xi.
\end{aligned}
\]
Then, we get
\[
\sigma(\xi)=\|w_{*}(s-t_{*}+\xi)-w_{*}(\xi)\|\leq \operatorname{Lip}(w_{*})(s-t_{*})
\]
when \(\xi \in [-\bar\tau,t_{*}-s]\) and 
\[
\begin{aligned}
\sigma(\xi)=\left\|
        z_{*}+\frac{x-z_{*}}{s-t_{*}}(s-t_{*}+\xi)-w_{*}(\xi)
    \right\|&\le
\left\|
        w_{*}(0)
        -
        w_{*}(\xi)
\right\|
    +\frac{\|x-z_{*}\|}{s-t_{*}}|s-t_{*}+\xi|\\
    &\le \operatorname{Lip}(w_{*})|\xi|+C_*|s-t_{*}+\xi|\\
    &\le (\operatorname{Lip}(w_{*})+C_*)(s-t_{*})
\end{aligned}
\]
when \(\xi \in [t_{*}-s,0]\). Thus, we obtain
\[
    \|v^{s,x}_s-w_{*}\|_1 \le C(s-t)
\]
for \(C=C(\operatorname{Lip}(w_{*}),C_*)\).

We now prove \eqref{second_estimate}. Assume first that \(r\ge s\). Then, we have
\begin{equation*}
    \|u_r-v_s\|_1 =
    \int_{-\bar\tau}^0 \|u ^{r,y}(r+\xi)-v^{s,x}(s+\xi)\|\,d\xi =:
    \int_{-\bar\tau}^0 \gamma(\xi)\,d\xi,
\end{equation*}
where
\begin{equation}\label{eq:ur_vs}
\gamma(\xi)=
\begin{cases}
\|w_{*}(r-t_{*}+\xi)-w_{*}(s-t_{*}+\xi)\|, 
& \xi\in[-\bar{\tau},\,t_{*}],\\[3pt]
\left\|
        z_{*}+\dfrac{y-z_{*}}{r-t_{*}}(r-t_{*}+\xi)
        -
        w_{*}(s-t_{*}+\xi)
    \right\|,& \xi \in [t_{*}-r,t_{*}-s],\\[8pt]
\left\|
        \dfrac{y-z_{*}}{r-t_{*}}(r-t_{*}+\xi)
        -\dfrac{x-z_{*}}{s-t_{*}}(s-t_{*}+\xi)
    \right\|,
& \xi\in[t_{*}-s,\,0].
\end{cases}
\end{equation}
As for the first estimate, we have
\[
  \gamma(\xi)\le \operatorname{Lip}(w_{*}) (r-s)
\]
for \(\xi\in[-\bar{\tau},\,t_{*}-r]\) and
\[
 \gamma(\xi)\le \operatorname{Lip}(w_{*})|s-t_{*}+\xi|+C_*|r-t_{*}+\xi|\le (\operatorname{Lip}(w_{*})+C_*)(r-s)
\]
for \(\xi \in [t_{*}-r,t_{*}-s]\).

For \(\xi\in[t_{*}-s,\,0]\), observe that
\[
\begin{aligned} \frac{(y-z_{*})(r-t_{*}+\xi)}{r-t_{*}}
-
\frac{(x-z_{*})(s-t_{*}+\xi)}{s-t_{*}}  =
(y-x)
+
\frac{y-x}{r-t_{*}}\xi
-
\frac{(x-z_{*})(r-s)}{(s-t_{*})(r-t_{*})}\xi .
\end{aligned}
\]
Since \(|\xi|\le s-t_{*}\le r-t_{*}\), it follows that
\[
\begin{aligned}
\gamma(\xi)\le \|y-x\|+\frac{\|y-x\|}{r-t_{*}}|\xi|+C_*\frac{r-s}{r-t_{*}}|\xi|\le 2\|y-x\|+C_*(r-s).
\end{aligned}
\]
Combining the above estimates, we deduce that
\[
    \|u ^{r,y}_r-v^{s,x}_s\|_1
    \le C(r-s+\|y-x\|)
\]
for \(C=C(\operatorname{Lip}(w),C_*,\bar{\tau})\). The case \(r<s\) follows by changing the roles of \((r,y)\) and \((s,x)\).

To conclude, we show \eqref{final_estimate}. For this, we write
\[
\begin{aligned}
&\|u ^{r,y}_r(-\tau(r,y))-v^{s,x}_s(-\tau(s,x))\| \\
&\le
\|u ^{r,y}_r(-\tau(r,y))-v^{s,x}_s(-\tau
(r,y))\|
+
\|v^{s,x}_s(-\tau
(r,y))-v^{s,x}_s(-\tau(s,x))\|.
\end{aligned}
\]
The first term on the right-hand side is exactly the estimate \eqref{eq:ur_vs}. Hence,
\[
\|u ^{r,y}_r(-\tau(r,y))-v^{s,x}_s(-\tau
(r,y))\|\le C(|r-s|+\|y-x\|),
\]
where $C=C(\operatorname{Lip}(w_{*}),C_{*})$. By the Lipschitz continuity of \(v\), \(w_*\) and \(\tau\), the second term is also controlled by $(|r-s|+\|y-x\|)$ with a constant $C$ depending additionally on $\left\Vert z _{*}\right\Vert$.
\end{proof}

\begin{lemma}\label{lem:bar_phi_continuous}
Let $\varphi \in \Phi$, and \(\bar\varphi\) be defined as in \eqref{eq:bar_phi} with base point $(t_{*},w_{*})$. Then \(\bar\varphi\) is continuous on \(K\), the cone defined in \eqref{kone}.
\end{lemma}

\begin{proof}
Let \((s,x)\to (r,y)\) in \(K\). We first consider the case
\((r,y)=(t_{*},z_{*})\). 

Let \(v^{s,x}\in\Lambda_0(t_{*},w_{*};s,x)\). By the local
Lipschitz continuity of \(\varphi\), we have
\[
\begin{aligned}
    |\bar\varphi(s,x)-\bar\varphi(t_{*},z_{*})|
    &=
    |\varphi(s,v_s)-\varphi(t_{*},w_{*})|  \\
    &\le
    \lambda_\varphi
    \bigl(
        |s-t_{*}|+\|x-z_{*}\|+\|v_s-w\|_1
    \bigr).
\end{aligned}
\]
Since \((s,x)\in K\), we have \(\|x-z_{*}\|\le C_*(s-t_{*})\), and \Cref{lem:continuous_polygon_extension} gives
\[
    \|v^{s,x}_s-w\|_1\le C(s-t_{*}).
\]
Hence \(\bar\varphi(s,x)\to\bar\varphi(t_{*},z_{*})\).

For a general point \((r,y)\in K\), set \(u ^{r,y}\in\Lambda_0(t_{*},w_{*};r,y)\). Again by the local Lipschitz continuity of \(\varphi\), we deduce
\[
    |\bar\varphi(s,x)-\bar\varphi(r,y)|
    \le
    \lambda_\varphi
    \bigl(
        |s-r|+\|x-y\|+\|v^{s,x}_s-u^{r,y}_r\|_1
    \bigr).
\]
Then, \Cref{lem:continuous_polygon_extension} implies the continuity of \(\bar\varphi\) on \(K\).
\end{proof}

 Let \(\psi_1\in\Phi\) be a viscosity subsolution and \(\psi_2\in\Phi\) be a viscosity supersolution of \eqref{eq:HJE}. Set
\[
    h_{*}:=T_{*}-t_{*}<1/(2C_{f}),\qquad m_*:=\|w_*\|_\infty,
    \qquad
    C_*:=
    \frac{1+2m_*}{1/C_{f}-2h_*},
\]
where $C_f$ is taken from \Cref{ass:standing}. With the base point $(t_{*},w_{*})$, all polygonal extensions in the cone $K$ (defined in \eqref{kone}) belong to $V(R)$ for some $R=R(m_{*},C_{*},h_{*})$. Set 
\[
  \lambda _{*}=\max_{}\left\{\lambda _{\psi _{1}}(R),\lambda _{\psi _{2}}(R)\right\} ,\quad \theta =2\lambda _{*}C_{*}h_{*},\quad \eta =2\theta.
\]

\textcolor{red}{We next establish a comparison estimate on a sufficiently short time interval. A preliminary difficulty arises from the fact that the doubling-of-variables argument requires changing the base point of the polygonal projections. \Cref{lem:rebase} shows that the error produced by this change of base point is controlled by the corresponding time increment. 
Let \(\psi_1\in\Phi\) be a viscosity subsolution and \(\psi_2\in\Phi\) a viscosity supersolution of \eqref{eq:HJE}. Fix \(t_*<T_*\) such that \[ h_*:=T_*-t_*<\frac{1}{2C_f}, \] and let \[ m_*:=\|w_*\|_\infty, \qquad C_*:= \frac{1+2m_*}{1/C_f-2h_*}, \] where \(C_f\) is the constant appearing in \Cref{ass:standing}. With \((t_*,w_*)\) as the base point, all polygonal extensions associated with points in the cone \[ K:=K(t_*,T_*;w_*(0),C_*) \] belong to \(V(R)\) for some constant \[ R=R(m_*,C_*,h_*). \] We then set \[ \lambda_*:= \max\left\{ \lambda_{\psi_1}(R), \lambda_{\psi_2}(R) \right\}, \qquad \theta:=2\lambda_*C_*h_*, \qquad \eta:=2\theta. \] Unless otherwise specified, throughout the following argument \(K\) denotes the cone \(K(t_*,T_*;w_*(0),C_*)\) defined above.}

\begin{lemma}[Rebase estimate]\label{lem:rebase}
Let
\((r,y),(s,x)\in K\) with $r<s$ and $\left\Vert x-y \right\Vert \leq C_{*}(s-r)$. Set
\[
    u ^{r,y}\in \Lambda_0(t_*,w_*;r,y),
    \qquad
    v^{s,x}\in \Lambda_0(t_*,w_*;s,x).
\]
Further, let $q ^{s,x}\in \Lambda _{0}(r,u_{r}; s,x)$. Then, we have
\[
  \left|\psi _{i}(s,q^{s,x}_{s})-\psi _{i}(s,v^{s,x}_{s})\right|\leq \theta (s-r), \quad i=1,2.
\]
\end{lemma}

\begin{proof}
It is easy to see
\[
  \left\Vert q^{s,x}(\xi )-v^{s,x}(\xi ) \right\Vert\leq \left\Vert q^{s,x}(r)-v^{s,x}(r) \right\Vert 
\]
for all $\xi \in [t_{*},s]$. While
\[
  \left\Vert q^{s,x}(r)-v^{s,x}(r) \right\Vert=\left\Vert y-z_{*}-\frac{x-z_{*}}{s-t_{*}} (r-t_{*})  \right\Vert \leq \left\Vert y-x \right\Vert +\frac{\left\Vert x-z_{*} \right\Vert }{s-t_{*}} (s-r)\leq 2C_{*}(s-r).
\]
Then, we get
\[
  \left\Vert q^{s,x}_{s}-v^{s,x}_{s} \right\Vert _{1}\leq 2C_{*}h_{*}(s-r).
\]
The claim follows from
\[
  \left|\psi _{i}(s,q^{s,x}_{s})-\psi _{i}(s,v^{s,x}_{s})\right|\leq \lambda _{*}\left\Vert q^{s,x}_{s}-v^{s,x}_{s} \right\Vert_{1}\quad i=1,2
\]
and this completes the proof. 
\end{proof}

Combining the rebasing estimate with a doubling-of-variables argument, we obtain the following local comparison estimate.

\begin{lemma}\label{lem:short_time}
Let $\bar{\psi}_{i}$ be defined as in \eqref{eq:bar_phi} with base point $(t_{*},w_{*})$, for $i=1,2$. Set
\[
  S:=\max _{(T_{*}, x) \in K}\left\{\bar{\psi}_1(T_{*}, x)-\bar{\psi}_2(T_{*}, x)\right\} .
\]
Then, we have
\[
  \bar{\psi}_1\left(t_*, z_*\right)-\bar{\psi}_2\left(t_*, z_*\right) \leq S+8 \lambda _{*} C_* h_*^2.
\]
\end{lemma}

\begin{proof}
We argue by contradiction. With $\eta  =4\lambda _{*}C_{*}h_{*}$, we have 
\[
  \bar{\psi}_1\left(t_*, z_*\right)-\bar{\psi}_2\left(t_*, z_*\right)>S+2\eta (T_{*}-t_{*}).
\]
By the continuity of $\bar{\psi}_{i}$, we can find $(\bar{t} ,\bar{z})\in K$ with $t_{*}<\bar{t}<T_{*}$ and $\bar{z}=z_{*} $ such that 
\begin{equation}\label{tbar_zbar}
  \bar{\psi}_1\left(\bar{t} , \bar{z} \right)-\bar{\psi}_2\left(\bar{t} ,\bar{z} \right)-2\eta (T_{*}-\bar{t})>S.
\end{equation}
Choose
\begin{equation}\label{eq:M_big}
    M>
    \sup_{(t,z),(s,x)\in K}
    |\bar\psi_1(t,z)-\bar\psi_2(s,x)|.
\end{equation}
Let \(\chi\in C^1(\mathbb R)\) be nonincreasing and such that
\[
    \chi(r)=0 \quad\text{for } r\le -\delta,
    \qquad
    \chi(r)=-3M \quad\text{for } r\ge 0,
\]
where $0<\delta<C_{*}(\bar{t}-t_{*})$. For \(\varepsilon,\eta,D>0\) and \((t,z),(s,x)\in K\), define
\[
    \langle z\rangle_D:=\bigl(\|z-z_*\|^2+D^2\bigr)^{1/2}
\]
and
\[
\begin{aligned}
    \Psi(t,z,s,x)
    &:=
    \bar\psi_1(t,z)-\bar\psi_2(s,x)
    -
    \frac{\|z-x\|^2+|t-s|^2}{2\varepsilon}
    -
    \eta(2T_{*}-t-s) \\
    &\quad
    +
    \chi\bigl(\langle z\rangle_D-C_*(t-t_*)\bigr)
    +
    \chi\bigl(\langle x\rangle_D-C_*(s-t_*)\bigr).
\end{aligned}
\]
Since \(K\) is compact and
\(\bar\psi_1,\bar\psi_2\) are continuous, \(\Psi\) attains its maximum at
some point \((\hat t,\hat z,\hat s,\hat x)\in K\times K\).

Note that
\[
    \Psi(\hat t,\hat z,\hat t,\hat z)
    +
    \Psi(\hat s,\hat x,\hat s,\hat x)
    \le
    2\Psi(\hat t,\hat z,\hat s,\hat x),
\]
which implies
\begin{equation}\label{eq:doubling_basic}
    \frac{\|\hat z-\hat x\|^2+|\hat t-\hat s|^2}{\varepsilon}
    \le
    \bar\psi_1(\hat t,\hat z)-\bar\psi_1(\hat s,\hat x)
    +
    \bar\psi_2(\hat t,\hat z)-\bar\psi_2(\hat s,\hat x).
\end{equation}
Since $\bar{\psi} _{i}$ are bounded on the cone, we have
\begin{equation}\label{eq:doubling_bound}
    \|\hat z-\hat x\|^2+|\hat t-\hat s|^2
    \le
    C\varepsilon.
\end{equation}
Moreover, the continuity of \(\bar\psi_1\) and \(\bar\psi_2\), \eqref{eq:doubling_basic} and \eqref{eq:doubling_bound} yield
\begin{equation}\label{eq:doubling_modulus}
    \frac{\|\hat z-\hat x\|^2+|\hat t-\hat s|^2}{\varepsilon}
    \le
    o(1) \quad \text{as } \varepsilon\downarrow 0.
\end{equation}

Furthermore, we will show that
\begin{equation}\label{zhat_xhat}
  \left\Vert \hat{z} -\hat{x}  \right\Vert /\varepsilon \leq C_{0}
\end{equation}
for some $C_{0}>0$. We take the case $\hat{t}\leq \hat{s}$. Denote $B(t,z)=\chi\bigl(\langle z\rangle_D-C_*(t-t_*)\bigr)$. Then the interpolation point $(\hat{s},\hat{z})$ is also in $K$. Then we have
\[
  \Psi (\hat{t} ,\hat{z},\hat{s},\hat{z}  )\leq \Psi(\hat t,\hat z,\hat s,\hat x),
\]
which gives
\[
 \left\Vert \hat{z} -\hat{x}  \right\Vert^{2} /(2\varepsilon)\leq  \bar{\psi}_{2}(\hat{s} ,\hat{z})- \bar{\psi}_{2}(\hat{s},\hat{x})+B(\hat{s},\hat{x})-B(\hat{s},\hat{z}).
\]
The case $\hat{t}>\hat{s}$ can be tackled by comparing $\Psi (\hat{t} ,\hat{x},\hat{s},\hat{x}  )$ and $\Psi(\hat t,\hat z,\hat s,\hat x)$. Then the claim follows from the Lipschitz properties of $\psi_{2}$ and $B$ and \Cref{lem:continuous_polygon_extension}.

We next show that the maximum point \((\hat t,\hat z,\hat s,\hat x)\) lies in the interior of $K\times K$. Choose \(D>0\) sufficiently small such that $\langle \bar z\rangle_D-C_*(\bar t-t_*)<-\delta$. By \eqref{tbar_zbar} and the construction of $\chi $, we obtain
\[
    \Psi(\hat t,\hat z,\hat s,\hat x)
    \ge
    \Psi(\bar t,\bar z,\bar t,\bar z)>S.
\]
If \(\hat t=T_{*}\), the continuity of \(\bar{\psi} _2\) and \eqref{eq:doubling_bound} imply
\[
\begin{aligned}
    \Psi(\hat t,\hat z,\hat s,\hat x)
    &\le
    \bar\psi_1(T_{*},\hat z)-\bar\psi_2(\hat s,\hat x)  \\
    &\le
    \bar\psi_2(T_{*},\hat z)-\bar\psi_2(\hat s,\hat x)+S\to S \text{ as } \varepsilon\downarrow 0,
\end{aligned}
\]
which contradicts the lower bound \(\Psi(\hat t,\hat z,\hat s,\hat x)> S\) for small \(\varepsilon\). The same argument gives
\(\hat s<T_{*}\).
If
\[
    \|\hat z-z_*\|=C_*(\hat t-t_*),
    \qquad
    \text{or}
    \qquad
    \|\hat x-z_*\|=C_*(\hat s-t_*),
\]
then one of the two penalization terms from $\chi$ equals \(-3M\). While \eqref{eq:M_big} implies
\[
  \Psi(\hat t,\hat z,\hat s,\hat x)\le -2M,
\]
a contradiction since $M>\left|S\right| $.

Before we use the viscosity inequalities, we prove some estimates for $H$. If \((t,z)\in K\) and \(\|y\|\le m_*+C_*h_*\), then
\begin{equation}\label{cone_bdd_C_*}
\begin{aligned}
    C_{f}(1+\|z\|+\|y\|)
    &\le
    C_{f}(1+2m_*+2C_*h_*)=C_*.
\end{aligned}
\end{equation}
Hence, by \Cref{lem:H_lipschitz_p},
\begin{equation}\label{eq:H_lipschitz_p_cone}
    |H(t,z,y,p)-H(t,z,y,p')|
    \le
    C_*\|p-p'\|
\end{equation}
for all \((t,z)\in K\), all admissible delayed states \(y\) arising from
the polygonal histories in \(K\), and all \(p,p'\in\mathbb R^n\). Set
\[
\begin{gathered}
Z=\langle\hat{z}\rangle_D-C_*\left(\hat{t}-t_*\right), \quad X=\langle\hat{x}\rangle_D-C_*\left(\hat{s}-t_*\right), \\
p_1=\frac{\hat{z}-\hat{x}}{\varepsilon}-\chi^{\prime}(Z) \frac{\hat{z}-z_*}{\langle\hat{z}\rangle_D}, \\
p_2=\frac{\hat{z}-\hat{x}}{\varepsilon}+\chi^{\prime}(X) \frac{\hat{x}-z_*}{\langle\hat{x}\rangle_D} 
\end{gathered}
\]
and
\[
\begin{array}{ll}
Y_1=u _{\hat{t}}(-\tau(\hat{t}, \hat{z})), & u \in \Lambda_0\left(t_*, w_* ; \hat{t}, \hat{z}\right), \\[2mm]
Y_2=v _{\hat{s}}(-\tau(\hat{s}, \hat{x})), & v \in \Lambda_0\left(t_*, w_* ; \hat{s}, \hat{x}\right) .
\end{array}
\]
By \eqref{eq:doubling_bound}, \eqref{zhat_xhat}, \eqref{eq:H_lipschitz_p_cone}, \Cref{lem:continuous_polygon_extension}, and the uniform continuity of $H$ on compact subsets, we have
\[
  H\left(\hat{t}, \hat{z}, Y_1, p_1\right)-H\left(\hat{s}, \hat{x}, Y_2, p_2\right) \leq \theta+C_*\left\|p_1-p_2\right\| .
\]
for small $\varepsilon $.

Define 
\[
  \phi _{1}(t,z)=-\Psi (t,z,\hat{s},\hat{x})+\bar{\psi }_{1}(t,z)+\theta (t-\hat{t}), \quad \phi _2(s,x)=\Psi (\hat{t} ,\hat{z},s,x)+\bar{\psi }_{2}(s,x)-\theta (s-\hat{s})
\]
for $(t,z),(s,x)\in \mathbb{R}\times \mathbb{R}^{n}$. Then $\phi _{i}\in C^{1}(\mathbb{R}\times \mathbb{R}^{n})$ for $i=1,2$. Let $(\hat{t},\hat{z})$ and $(\hat{s},\hat{x})$ be new base points. Define
\[
  \hat{\psi}_{1}(t,z)=\psi _{1}(t,q^{t,z}_{t}),\quad q^{t,z}\in \Lambda _{0}(\hat{t}, u_{\hat{t}};t,z), \quad (t,z)\in K,\quad  t>\hat{t}.
\]
Set
\[
  v ^{t,z} \in \Lambda _{0}(t_{*},w_{*}; t,z).
\]
Then \eqref{cone_bdd_C_*} and \Cref{lem:rebase} imply
\[
\begin{aligned}
  \hat{\psi}_{1}(t,z)-\phi_{1}(t,z)&=\psi _{1}(t,q^{t,z}_{t})-\psi _{1}(t,v^{t,z}_{t})-\theta (t-\hat{t})+\Psi(t,z,\hat{s},\hat{x})\\
  &\leq \Psi(t,z,\hat{s},\hat{x})\leq \Psi (\hat{t} ,\hat{z} ,\hat{s} ,\hat{x} )\\
  &=\bar{\psi}_{1}(\hat{t} ,\hat{z})-\phi_{1}(\hat{t} ,\hat{z}) \\
  &= \hat{\psi}_{1}(\hat{t} ,\hat{z})-\phi_{1}(\hat{t} ,\hat{z})
\end{aligned}
\]
for all $(t,z)\in K(\hat{t},T_{*},\hat{z},C_{*})$. That is, $\hat{\psi}_{1}-\phi _{1}$ attains local maximum at $(\hat{t} ,\hat{z})$. Similarly we can define $\hat{\psi}_{2}$ with base point $(\hat{s},v_{\hat{s}})$ and show that $\hat{\psi}_{2}-\phi_{2}$ attains local minimum at $(\hat{s},\hat{x})$. We note that
\[
C_{f}(1+2\left\Vert u_{\hat{t} } \right\Vert_{\infty})\leq C_{f}(1+2m_{*}+2C_{*}(\hat{t}-t_{*}))<C_{*}.
\]
Thus, the viscosity inequalities yields
\[
\begin{aligned}
& \frac{\hat{t}-\hat{s}}{\varepsilon}-\eta+\theta+C_* \chi^{\prime}(Z)+H\left(\hat{t}, \hat{z}, Y_1, p_1\right) \geq 0, \\
& \frac{\hat{t}-\hat{s}}{\varepsilon}+\eta-\theta-C_* \chi^{\prime}(X)+H\left(\hat{s}, \hat{x}, Y_2, p_2\right) \leq 0.
\end{aligned}
\]
Subtracting the above inequalities, by \eqref{eq:H_lipschitz_p_cone} and $\chi '\leq 0$, we obtain
\[
\begin{aligned}
0 & \leq C_*\left(\chi^{\prime}(Z)+\chi^{\prime}(X)\right)-2\theta +H\left(\hat{t}, \hat{z}, Y_1, p_1\right)-H\left(\hat{s}, \hat{x}, Y_2, p_2\right) \\
& \leq C_*\left(\chi^{\prime}(Z)+\left|\chi^{\prime}(Z)\right|+\chi^{\prime}(X)+\left|\chi^{\prime}(X)\right|\right)-\theta \\
& \leq-\theta,
\end{aligned}
\]
which yields a contradiction.
\end{proof}

\textcolor{red}{The comparison estimate established above is local in time, it requires the length of the time interval to be sufficiently small. To derive a global comparison principle, we divide the interval \([t_{1},T]\) into \(N\) subintervals of equal length, chosen so that \Cref{lem:short_time} applies on each of them. We then propagate the comparison estimate backward from the terminal time and concatenate the resulting local estimates. Since the total accumulated error is of order \(N^{-1}\), it vanishes as \(N\to\infty\), leading to the following global result.}

\begin{theorem}\label{thm:comparison}
If $\psi_{1}(T,\cdot)\leq \psi _{2}(T,\cdot)$, then $\psi _{1}\leq \psi _{2}$. In particular, if $\psi _1$ and $\psi _{2}$ are two viscosity solutions of \eqref{eq:HJE} with the terminal condition \eqref{eq:terminal}, then $\psi _{1}=\psi _{2}$.
\end{theorem}
\begin{proof}
Let $\Delta \psi :=\psi _{1}-\psi _{2}$. Assume by contradiction that there exists $(t_{1},w_{1})\in G$ with $t_{1}<T$ such that 
\[
  \Delta \psi (t_{1},w_{1})=\sigma >0.
\]
Let $T_{1}=T-t_{1}$. Choose $N\in \mathbb{N}$ such that $h:=T_{1}/N<1/(4C_{f})$. Set $t_{k}=t_{1}+(k-1)h$ for $k=1,\ldots ,N+1 $. Then $t_{N+1}=T$. Define iterative constants
\[
  C_k:=\frac{1+2 B_k}{1 / C_f-2 h}, \quad B_{k+1}:=B_k+C_k h, \quad  B_{1}:=\left\Vert w_{1} \right\Vert _{\infty}, \quad k=1,\ldots ,N.
\]
Then, we can write
\[
  C_k=C_f\left(1+2 B_{k+1}\right), \quad 1+2 B_{k+1}=\frac{1+2 B_k}{1-2 C_f h}.
\]
Thus there exist constants $\bar{B}$ and $\bar{C}$ independent of $N$ such that $B_{k}\leq \bar{B}$ and $C_{k}\leq \bar{C} $.

Denote $z_{1}=w_{1}(0)$. We fix the base point $(t_{1},z_{1})$ and define
\[
  \bar{\psi}^{1}_{i}(t,z)=\psi _{i}(t,u^{t,z}_{t}), \quad u ^{t,z}\in \Lambda_{0}(t_{1},w_{1}; t,z),\quad i=1,2
\]
for $(t,z)\in K(t_{1},t_{2},z_{1},C_{1})$. Let $z_{2}$ be the maximum point of $\bar{\psi}^{1}_{1}(t_{2},\cdot)-\bar{\psi}^{1}_{2}(t_{2},\cdot)$. Then set $w_{2}=u^{t_{2},z_{2}}_{t_{2}}$. Consecutively we can define $\bar{\psi}^{k}_{i}$ and $(z_{k},w_{k})$ for $k=1,\ldots ,N$. Then all relevant polygonal histories are contained in some $V(R)$. Let $\Delta \bar{\psi} ^{k}:=\bar{\psi}_{1}^{k}-\bar{\psi}_{2}^{k}$. Since $S\leq 0$, we apply \Cref{lem:short_time} on $[t_{N},t_{N+1}]$ to obtain
\[
  \Delta \bar{\psi}^{N}(t_{N},z_{N})\leq 8\lambda _{*}C_{N}h^{2}.
\]
Then we repeat it backward in time to find
\[
  \Delta \bar{\psi}^{k}(t_{k},z_{k})\leq \Delta \bar{\psi}^{k+1}(t_{k+1},z_{k+1}) +8\lambda _{*}C_{k}h^{2}, \quad k=1,\ldots ,N-1.
\]
These iterative inequalities yield
\[
\Delta \bar{\psi}^{1}(t_{1},z_{1})\leq 8\lambda _{*}\sum_{k=1}^{N}C_{k}h^{2}\leq 8\lambda _{*}\bar{C} T_{1}^{2}/N.
\]
Note that $\Delta \bar{\psi}^{1}(t_{1},z_{1})=\Delta \psi(t_{1},w_{1})=\sigma$. As $N\to \infty$, we obtain a contradiction. The uniqueness follows from changing the roles of $\psi _{1}$ and $\psi _{2}$.
\end{proof}

\section{Optimality conditions}\label{sec:optimality}

In this section we derive first-order optimality conditions for optimal trajectories. We assume, in addition to the standing assumptions, that \(f\) and \(L\) are \(C^1\) in \((z,y)\) uniformly with respect to \((t,a)\). We use \(f_z\), \(f_y\) to denote the matrix functions \((\frac{\partial f^i}{\partial x_j})_{1\le i,j\le n}\) and \((\frac{\partial f^i}{\partial y_j})_{1\le i,j \le n}\) respectively. We view all \(n\)-dimensional vectors as column vectors.



\begin{lemma}[Linearization of the state equation]\label{lem:linearization}
Let \((t,w)\in\mathbb G\), \(\alpha\in \mathcal{A}_{t,T}\), and \(x(\cdot)=x(\cdot\,|\,t,w,\alpha)\). We set
\[
    D(s):=s-\tau(s,x(s)),
    \qquad
    s\in[t,T]
\]
and denote its inverse by
\[
    \gamma:[D(t),D(T)]\to[t,T].
\]
Let \(R>0\) be fixed. Suppose that
\(\{w_\varepsilon\}_{\varepsilon>0}\subset V(R)\) satisfies
\[
    \operatorname{supp} w_\varepsilon\subset[-\varepsilon,0],
    \qquad
    \|w_\varepsilon\|_\infty=O(\varepsilon),
\]
and that there exists \(\bar v\in\mathbb R^n\) such that
\[
    w_\varepsilon(0)=\varepsilon\bar v+o(\varepsilon)
    \qquad
    \text{as }\varepsilon\downarrow0.
\]
Assume also that \(w\in V(R)\). Let
\[
    x_\varepsilon(s)
    :=
    x(s\,|\,t,w+w_\varepsilon,\alpha).
\]
Then
\[
    x_\varepsilon(s)
    =
    x(s)+\varepsilon v(s)+o(\varepsilon)
\]
uniformly for \(s\in[t,T]\), where \(v\) is the unique solution of the
linearized equation
\begin{equation}\label{eq:linearized_state}
\begin{aligned}
    \dot v(s)
    &=
    f_z(s,x(s),x(D(s)),\alpha(s))\,v(s)        \\
    &
    -
    f_y(s,x(s),x(D(s)),\alpha(s))
        \dot x(D(s))\,\tau_z(s,x(s))\cdot v(s)                                     \\
    &
    +
    f_y(s,x(s),x(D(s)),\alpha(s))\,v(D(s))
    \mathbf 1_{[t,D(T)]}(D(s))
\end{aligned}
\end{equation}
for a.e. \(s\in (t,T)\), with initial condition \(    v(t)=\bar v\).
\end{lemma}

\begin{proof}
It suffices to consider the case \(t<D(T)\), so that \(\gamma(t)\) is well-defined. The case \(t\geq D(T)\) is simpler, since the indicator term in \eqref{eq:linearized_state} vanishes.

Applying \Cref{lem:continuous_dependence}, together
with the localization of the perturbation \(w_\varepsilon\), we obtain
\begin{equation}\label{eq:xeps_x_order}
    \sup_{s\in[t,T]}\|x_\varepsilon(s)-x(s)\|
    =
    O(\varepsilon).
\end{equation}
For \(s\in [t,T]\), set
\[
    D_\varepsilon(s):=s-\tau(s,x_\varepsilon(s))
\]
and define
\[
    y_\varepsilon(s)
    =
    \frac{x_\varepsilon(s)-x(s)}{\varepsilon}.
\]
Subtracting the integral equation for \(x\) from the one for \(x_\varepsilon\), and dividing the resulting equation by \(\varepsilon\), we then deduce from the
\(C^{1}\)-regularity of \(f\) in \((z,y)\) that
\begin{equation}\label{eq:increment_equation}
\begin{aligned}
    y_\varepsilon(s)
    &=
    \bar v
    +
    \int_t^s
    f_z(\eta,x(\eta),x(D(\eta)),\alpha(\eta))
    y_\varepsilon(\eta)\,d\eta                    \\
    &\quad
    +
    \frac1\varepsilon
    \int_t^s
    f_y(\eta,x(\eta),x(D(\eta)),\alpha(\eta))
    \bigl[
        x_\varepsilon(D_\varepsilon(\eta))-x(D(\eta))
    \bigr]\,d\eta
    +
    o(1).
\end{aligned}
\end{equation}
In view of \eqref{eq:xeps_x_order}, the Lipschitz continuity of \(\tau\), and the localization of \(w_\varepsilon\), the remainder term in the expansion of \(f\) is \(o(\varepsilon)\) in \(L^1\). We decompose
\[
\begin{aligned}
    x_\varepsilon(D_\varepsilon(\eta))-x(D(\eta))
    &=
    \bigl[
        x_\varepsilon(D_\varepsilon(\eta))
        -
        x(D_\varepsilon(\eta))
    \bigr]  
    \\
    &\quad
    +
    \bigl[
        x(D_\varepsilon(\eta))-x(D(\eta))
    \bigr] 
    \\
    &=:A_\varepsilon(\eta)+B_\varepsilon(\eta).
\end{aligned}
\]

Let \(\gamma_\varepsilon\) be the inverse of \(D_\varepsilon\). We may also assume that \(t<D_\varepsilon(T)\) so that \(\gamma_\varepsilon(t)\) is well-defined. Suppose that \(\gamma_\varepsilon(t)<\gamma(t)\). The other case can be dealt with similarly. Since
\(\dot{D}\ge d(2R)>0\), \eqref{eq:xeps_x_order} implies
\begin{equation}\label{eq:gamma_eps}
    |\gamma_\varepsilon(t)-\gamma(t)|
    =
    O(\varepsilon).
\end{equation}

First we discuss the non-delayed case, i.e., \(\eta\le \gamma(t)\). By the localization of \(w_\varepsilon\), \eqref{eq:xeps_x_order} and \eqref{eq:gamma_eps}, we obtain
\[
    \int_t^{\gamma(t)} \|A_\varepsilon(\eta)\|\,d\eta
    = \int_t^{\gamma_\varepsilon(t)} \|A_\varepsilon(\eta)\|\,d\eta+ \int_{\gamma_\varepsilon(t)}^{\gamma(t)} \|A_\varepsilon(\eta)\|\,d\eta=
    o(\varepsilon).
\]
For a.e. \(\eta \in [t,\gamma(t)]\), since \(x\) is Lipschitz and \(D(\eta)\le t\), by Rademacher's theorem and \eqref{eq:xeps_x_order}, we have
\[
\begin{aligned}
B_\varepsilon(\eta)
&=
\dot w(D(\eta)-t)
\bigl(D_\varepsilon(\eta)-D(\eta)\bigr)
+
R_\eta(\varepsilon) \\
&=
-\dot w(D(\eta)-t)
\bigl[
    \tau(\eta,x_\varepsilon(\eta))
    -
    \tau(\eta,x(\eta))
\bigr]
+
R_\eta(\varepsilon) \\
&=
-\dot w(D(\eta)-t)
\bigl[
    \tau_z(\eta,x(\eta))\cdot
    (x_\varepsilon(\eta)-x(\eta))
\bigr]
+
R_\eta(\varepsilon),
\end{aligned}
\]
where \(R_\eta(\varepsilon)=o(\varepsilon)\) depends on the choice of \(\eta\). Note that
\[
|R_\eta(\varepsilon)|
=
|B_\varepsilon(\eta)-\dot w(D(\eta)-t)
\bigl(D_\varepsilon(\eta)-D(\eta)\bigr)|\le C \varepsilon.
\]
Dominated convergence theorem yields
\[
\lim_{\varepsilon\downarrow 0}\frac{1}{\varepsilon}\int_t^{\gamma(t)}|R_\eta(\varepsilon)|d\eta=0.
\]
Consequently, for \(s\in[t,\gamma(t)]\), \eqref{eq:increment_equation} becomes
\[
\begin{aligned}
    y_\varepsilon(s)
    &=
    \bar v
    +
    \int_t^s
    f_z(\eta,x(\eta),x(D(\eta)),\alpha(\eta))
    y_\varepsilon(\eta)\,d\eta                    \\
    &\quad
    -
    \int_t^s
    f_y(\eta,x(\eta),x(D(\eta)),\alpha(\eta))
        \dot w(D(\eta)-t)\,
        \tau_z(\eta,x(\eta))\cdot y_\varepsilon(\eta)\,d\eta
    +
    o(1).
\end{aligned}
\]
Subtracting the integral equation for \(v\) from this one for \(y_\varepsilon\), and applying Gronwall's
inequality, we deduce
\begin{equation}\label{eq:error_before_gamma}
    \sup_{s\in[t,\gamma(t)]}
    \|y_\varepsilon(s)-v(s)\|
    =
    o(1) \qquad \text{as } \varepsilon\downarrow 0.
\end{equation}

Now we turn to the delayed case, that is, \(\eta\ge \gamma(t)\).
The expansion arguments for \(B_\varepsilon\) still work. Thus, we have
\[
\begin{aligned}
B_\varepsilon(\eta)
&=
-\dot x(D(\eta))
\bigl[
    \tau_z(\eta,x(\eta))\cdot
    (x_\varepsilon(\eta)-x(\eta))
\bigr]
+
o(\varepsilon)
\end{aligned}
\]
for a.e. \(\eta \in [\gamma(t),T]\). For \(\eta \ge \gamma(t)\), we can write
\[
    \frac{A_\varepsilon(\eta)}{\varepsilon}
    =
    y_\varepsilon(D(\eta))+r_\varepsilon(\eta),
\]
where \(r_\varepsilon(\eta)=y_\varepsilon(D_\varepsilon(\eta))-y_\varepsilon(D(\eta))\). Since \(D(\eta),D_\varepsilon(\eta)\geq t\) and \(\|\dot{y}_\varepsilon\|\) is uniformly bounded on \([t,T]\) with respect to \(\varepsilon\), one can show that \(\|r_\varepsilon\|=o(1)\) by \eqref{eq:xeps_x_order}. Thus, for \(s>\gamma(t)\) we have
\[
\begin{aligned}
    y_\varepsilon(s)
    &=
    y_\varepsilon(\gamma(t))
    +
    \int_{\gamma(t)}^s
    f_z(\eta,x(\eta),x(D(\eta)),\alpha(\eta))
    y_\varepsilon(\eta)\,d\eta                    \\
    &\quad
    -
    \int_{\gamma(t)}^s
    f_y(\eta,x(\eta),x(D(\eta)),\alpha(\eta))
        \dot x(D(\eta))\,
        \tau_z(\eta,x(\eta))\cdot y_\varepsilon(\eta)\,d\eta                                 \\
    &\quad
    +
    \int_{\gamma(t)}^s
    f_y(\eta,x(\eta),x(D(\eta)),\alpha(\eta))
    y_\varepsilon(D(\eta))\,d\eta
    +
    o(1).
\end{aligned}
\]
Using
\eqref{eq:error_before_gamma}, and applying Gronwall's inequality for the difference of \(y_\varepsilon-v\) as before, we obtain
\[
    \sup_{s\in[t,T]}\|y_\varepsilon(s)-v(s)\|
    =
    o(1),
\]
which completes the proof.
\end{proof}
We next derive the Pontryagin minimum principle. We say that \(g\) is Fr\'echet differentiable at $(z,w)\in \mathcal{X}$ (see \eqref{domain_g} for the definition of $\mathcal{X}$) if there exist
\[
    g_z(z,w)\in\mathbb R^n,
    \qquad
    g_w(z,w)\in L^\infty([-\bar\tau,0];\mathbb R^n)
\]
such that
\[
\begin{aligned}
    g(z+\delta z,w+\delta w)-g(z,w)
    &=
    g_z(z,w)\cdot\delta z
    +
    \int_{-\bar\tau}^{0}
    g_w(z,w)(\theta)\cdot\delta w(\theta)\,d\theta +o(\left\Vert (\delta z,\delta w) \right\Vert _{\mathcal{X}} )
\end{aligned}
\]
for all $(\delta z,\delta w)\in \mathcal{X}$.

Define the pre-Hamiltonian \[\mathcal H(t,z,y,p,a):=f(t,z,y,a)\cdot p+L(t,z,y,a)\] for \((t,z,y,p,a)\in I\times \mathbb R^n\times \mathbb R^n\times \mathbb R^n\times A\).
\begin{theorem}[Pontryagin minimum principle]\label{thm:PMP}
Assume that
\(\alpha\in\mathcal{A}_{t_*,T}\) is an optimal control for \((t_*,w_*)\in \mathbb{G}\). Let \(   x(\cdot)=x(\cdot\,|\,t_*,w_*,\alpha)\). Assume also that \(g\) is Fr\'echet differentiable at \((x(T),x_{T})\). Define
\[
D(s)=s-\tau(s,x(s)), \qquad s\in [t_*,T]
\]
and let \(\gamma(\cdot)\) be its inverse. Let \(p:[t_*,T]\to\mathbb R^n\) solve the adjoint equation
\begin{equation}\label{eq:adjoint}
\begin{aligned}
\dot p(s)
&=
-f_z(s,x(s),x(D(s)),\alpha(s))^\top p(s)
-
L_z(s,x(s),x(D(s)),\alpha(s))       \\
&\quad
+
\tau_z(s,x(s))
\Bigl[
    \dot x(D(s))\cdot
    f_y(s,x(s),x(D(s)),\alpha(s))^\top p(s)
\Bigr]                              \\
&\quad
+
\tau_z(s,x(s))
\Bigl[
    \dot x(D(s))\cdot
    L_y(s,x(s),x(D(s)),\alpha(s))
\Bigr]                              \\
&\quad
-
\mathbf 1_{[t_*,D(T)]}(s)\dot\gamma(s)
f_y(\gamma(s),x(\gamma(s)),x(s),\alpha(\gamma(s)))^\top
p(\gamma(s))                        \\
&\quad
-
\mathbf 1_{[t_*,D(T)]}(s)\dot\gamma(s)
L_y(\gamma(s),x(\gamma(s)),x(s),\alpha(\gamma(s))) \\
&\quad
-
\mathbf 1_{[T-\bar\tau,T]}(s)
g_w(x(T),x_T)(s-T),\qquad \text{a.e. } s\in (t_*,T),
\end{aligned}
\end{equation}
with \(p(T)=g_z(x(T),x_T)\). Then
\begin{equation}\label{eq:PMP}
\begin{aligned}
    \mathcal H(t,x(t),x(D(t)),p(t),\alpha(t))&=
    \min_{a\in A}
    \mathcal H(t,x(t),x(D(t)),p(t),a) \qquad \text{a.e. } t\in (t_*,T].
\end{aligned}
\end{equation}
\end{theorem}

\begin{proof}
Let \(t\in(t_*,T)\) be a Lebesgue point for functions $s \mapsto f(s,x(s),x(D(s)),\alpha (s))$ and $s \mapsto L(s,x(s),x(D(s)),\alpha (s))$, and let \(a\in A\). For small \(\varepsilon>0\), define the
needle variation of \(\alpha\) by
\[
    \alpha_\varepsilon(s)
    :=
    \begin{cases}
        a, & s\in(t-\varepsilon,t],\\
        \alpha(s), & s\in (t_*,T]\setminus (t-\varepsilon,t].
    \end{cases}
\]
Let \(x_\varepsilon=x(t_*,w_*,\alpha_\varepsilon)\), and set \(\bar x:=x(t),\bar y:=x(D(t))\). Note that \(x_\varepsilon(s)=x(s)\) for \(s\le t-\varepsilon\). Standard estimates yield
\[
    \|x_\varepsilon(s)-x(s)\|,
    \,
    \|x_\varepsilon(s)-\bar x\|, 
    \,
    \|x_\varepsilon(D_\varepsilon(s))-\bar{y}\|
    =
    O(\varepsilon)
\]
for all \(s\in[t-\varepsilon,t]\). Furthermore, we have
\begin{equation}\label{eq:needle_initial_variation}
\begin{aligned}
    x_\varepsilon(t)-x(t)
    =
    \varepsilon
    \bigl[
        f(t,\bar x,\bar y,a)
        -
        f(t,\bar x,\bar y,\alpha(t))
    \bigr]
    +
    o(\varepsilon).
\end{aligned}
\end{equation}
Starting from \(t\), we regard \(x_{\varepsilon,t}\) as a perturbed history of \(x_t\) and apply \Cref{lem:linearization} to obtain
\begin{equation}\label{apply_linearization}
    x_\varepsilon(s)=x(s)+\varepsilon v(s)+o(\varepsilon),
    \qquad
    s\in[t,T],
\end{equation}
where \(v\) solves \eqref{eq:linearized_state} with
\[
    v(t)
    =
    f(t,\bar x,\bar y,a)
    -
    f(t,\bar x,\bar y,\alpha(t)).
\]

We next use the adjoint equation. We have
\[
\begin{aligned}
    p(T)\cdot v(T)-p(t)\cdot v(t)
    =
    \int_t^T
        \dot p(s)\cdot v(s)+p(s)\cdot\dot v(s)
    \,ds.
\end{aligned}
\]
Substituting \eqref{eq:linearized_state} and \eqref{eq:adjoint} into this equation, and changing variables, we derive
\begin{equation}\label{eq:adjoint_identity}
\begin{aligned}
    p(T)\cdot v(T)-p(t)\cdot v(t)
    &=
    -
    \int_t^T
    L_z(s,x(s),x(D(s)),\alpha(s))\cdot v(s)\,ds        \\
    &\quad
    +
    \int_t^T
    L_y(s,x(s),x(D(s)),\alpha(s))\cdot
        \dot x(D(s))\,\tau_z(s,x(s))\cdot v(s) \,ds                                         \\
    &\quad
    -
    \int_{t}^{T}
    \mathbf 1_{[t,D(T)]}(s)
    L_y(\gamma(s),x(\gamma(s)),x(s),\alpha(\gamma(s)))\cdot v(s)\dot{\gamma}(s)\,ds      \\
    &\quad
    -
    \int_{t\vee(T-\bar\tau)}^{T}
    g_w(x(T),x_T)(s-T)\cdot v(s)\,ds .
\end{aligned}
\end{equation}

On the other hand, using the same expansion argument for \(L\) as for \(f\) in the proof of \Cref{lem:linearization}, we obtain
\begin{equation}\label{eq:running_cost_expansion}
\begin{aligned}
    &\int_t^T
    \Bigl[
        L(s,x_\varepsilon(s),x_\varepsilon(D_\varepsilon(s)),\alpha(s))
        -
        L(s,x(s),x(D(s)),\alpha(s))
    \Bigr]\,ds                                      \\
    &=
    \varepsilon
    \int_t^T
    L_z(s,x(s),x(D(s)),\alpha(s))\cdot v(s)\,ds      \\
    &\quad
    -
    \varepsilon
    \int_t^T
    L_y(s,x(s),x(D(s)),\alpha(s))\cdot
        \dot x(D(s))\,\tau_z(s,x(s))\cdot v(s)\,ds                                      \\
    &\quad
    +
    \varepsilon
    \int_{t}^{T}
    \mathbf 1_{[t,D(T)]}(D(s)) L_y(s,x(s),x(D(s)),\alpha(s))\cdot v(D(s))\,ds
    +
    o(\varepsilon).
\end{aligned}
\end{equation}
Including the contribution to the needle interval and changing variables, we obtain
\[
\begin{aligned}
    &\int_{t_*}^{T}
    \Bigl[
        L(s,x_\varepsilon(s),x_\varepsilon(D_\varepsilon(s)),
          \alpha_\varepsilon(s))
        -
        L(s,x(s),x(D(s)),\alpha(s))
    \Bigr]\,ds                                      \\
    &=
    \varepsilon
    \bigl[
        L(t,\bar x,\bar y,a)
        -
        L(t,\bar x,\bar y,\alpha(t))
    \bigr]                                          \\
    &\quad
    +
    \varepsilon
    \int_t^T
    L_z(s,x(s),x(D(s)),\alpha(s))\cdot v(s)\,ds      \\
    &\quad
    -
    \varepsilon
    \int_t^T
    L_y(s,x(s),x(D(s)),\alpha(s))\cdot
        \dot x(D(s))\,\tau_z(s,x(s))\cdot v(s)\,ds                                      \\
    &\quad
    +
    \varepsilon
    \int_{t}^{T}
    \mathbf 1_{[t,D(T)]}(s) L_y(\gamma(s),x(\gamma(s)),x(s),\alpha(\gamma(s)))\cdot v(s)\dot{\gamma}(s)\,ds
    +
    o(\varepsilon).
\end{aligned}
\]
Using \eqref{eq:adjoint_identity}, this becomes
\[
\begin{aligned}
&\int_{t_*}^{T}
\Bigl[
L(s,x_\varepsilon(s),x_\varepsilon(D_\varepsilon(s)),\alpha_\varepsilon(s))
-
L(s,x(s),x(D(s)),\alpha(s))
\Bigr]\,ds \\
&=
\varepsilon
\bigl[
L(t,\bar x,\bar y,a)
-
L(t,\bar x,\bar y,\alpha(t))
\bigr] \\
&\quad
+
\varepsilon p(t)\cdot v(t)
-
\varepsilon p(T)\cdot v(T) \\
&\quad
-
\varepsilon
\int_{t\vee (T-\bar\tau)}^{T}
g_w(x(T),x_T)(s-T)\cdot v(s)\,ds
+
o(\varepsilon).
\end{aligned}
\]
By the Fr\'echet differentiability of \(g\) and \eqref{apply_linearization},
\[
\begin{aligned}
    g(x_\varepsilon(T),(x_\varepsilon)_T)-g(x(T),x_T)
    &=
    \varepsilon\,g_z(x(T),x_T)\cdot v(T)             \\
    &\quad
    +
    \varepsilon
    \int_{t\vee (T-\bar{\tau})}^{T}
    g_w(x(T),x_T)(s-T)\cdot v(s)\,ds
    +
    o(\varepsilon).
\end{aligned}
\]
Then the optimality of
\(\alpha\) gives
\[
    0
    \le
    \varepsilon
    \bigl[
        L(t,\bar x,\bar y,a)
        -
        L(t,\bar x,\bar y,\alpha(t))
        +
        p(t)\cdot v(t)
    \bigr]
    +
    o(\varepsilon),
\]
where the terminal terms cancel since \(p(T)=g_z(x(T),x_T)\). Dividing by \(\varepsilon\) and letting \(\varepsilon\downarrow0\) in the above inequality, we obtain
\[
\begin{aligned}
    &p(t)\cdot
    \bigl[
        f(t,x(t),x(D(t)),a)
        -
        f(t,x(t),x(D(t)),\alpha(t))
    \bigr]                                      \\
    &\quad
    +
    L(t,x(t),x(D(t)),a)
    -
    L(t,x(t),x(D(t)),\alpha(t))
    \ge0.
\end{aligned}
\]
Since \(a\in A\) is arbitrary, this proves \eqref{eq:PMP} for a.e. \(t\in(t_*,T]\).
\end{proof}

\begin{remark}
The adjoint equation is a backward time-dependent delay differential equation, which is well-posed and admits the unique Lipschitz solution.
\end{remark}

\begin{theorem}[Generalized transversality condition]
\label{thm:generalized_transversality}
Let \(p(\cdot)\) be as in \Cref{thm:PMP}. For any \(t\in (t_*,T)\) at which \(\alpha\) is continuous, we have
\[
    \bigl(-H(t,x(t),x(D(t)),p(t)),\,p(t)\bigr)
    \in
    \partial^{\mathrm{delay}}_+\rho(t,x_t).
\]
\end{theorem}
\begin{proof}
Let \(t\in(t_*,T]\) be a point at which \(\alpha\) is continuous. Let
\(\lambda>0\) and \(\theta\in\mathbb R^n\). For \(h>0\), let
\[
    \kappa\in
    \Lambda_0
    \bigl(
        t,x_t;
        t+\lambda h,
        x(t)+h(\lambda\dot x(t+)+\theta)
    \bigr).
\]
We shall estimate the upper Dini derivative of \(\rho\) in the direction
\((\lambda,\lambda\dot x(t+)+\theta)\).

First perturb the history before time \(t\) by a polygonal arc on
\([t-h,t]\). Let
\[
    e\in
    \Lambda_0
    \bigl(
        t-h,x_{t-h};
        t,x(t)+h\theta
    \bigr).
\]
Then
\[
    \|e_t-x_t\|_\infty=O(h).
\]
Let
\[
    x^h(\cdot):=x(\cdot\,|\,t,e_t,\alpha).
\]
By \Cref{lem:linearization},
\begin{equation}\label{x^h}
    x^h(s)=x(s)+h v(s)+o(h), \qquad s\in [t,T],
\end{equation}
where \(v\) solves
\eqref{eq:linearized_state} with \(v(t)=\theta\). Moreover, by \eqref{x^h}, and the right differentiability of \(x\) at \(t\), we have
\[
    \|\kappa_{t+\lambda h}-x^h_{t+\lambda h}\|_1=o(h),
    \qquad
    \|\kappa(t+\lambda h)-x^h(t+\lambda h)\|=o(h).
\]
Since \(\rho\in\Phi\), this implies
\[
\begin{aligned}
    &\partial^+\rho(t,x_t)
    \bigl(\lambda,\lambda\dot x(t+)+\theta\bigr)        \\
    &=
    \limsup_{h\downarrow0}
    \frac{
        \rho(t+\lambda h,\kappa_{t+\lambda h})
        -
        \rho(t,x_t)
    }{h}                                               \\
    &=
    \limsup_{h\downarrow0}
    \frac{
        \rho(t+\lambda h,x^h_{t+\lambda h})
        -
        \rho(t,x_t)
    }{h}.
\end{aligned}
\]

Using the admissible control \(\alpha\in \mathcal{A}_{t+\lambda h,T}\), we have
\[
\begin{aligned}
    \rho(t+\lambda h,x^h_{t+\lambda h})
    \le
    \int_{t+\lambda h}^{T}
    L(s,x^h(s),x^h(D^h(s)),\alpha(s))\,ds
    +
    g(x^h(T),x^h_T),
\end{aligned}
\]
where \(D^h(s):=s-\tau(s,x^h(s))\). On the other hand, the optimality of \(\alpha\) and the dynamic programming principle yield
\[
    \rho(t,x_t)
    =
    \int_t^T
    L(s,x(s),x(D(s)),\alpha(s))\,ds
    +
    g(x(T),x_T).
\]
Thus,
\[
\begin{aligned}
&\rho(t+\lambda h,x^h_{t+\lambda h})
        -
        \rho(t,x_t)\\
    \le & \int_{t}^{T}
    L(s,x^h(s),x^h(D^h(s)),\alpha(s))\,ds
   -\int_t^T
    L(s,x(s),x(D(s)),\alpha(s))\,ds\\
     &+
    g(x^h(T),x^h_T)
    -
    g(x(T),x_T)-\int_{t}^{t+\lambda h}
    L(s,x^h(s),x^h(D^h(s)),\alpha(s))\,ds.
\end{aligned}
\]
Expanding the difference of the integrals of \(L\) exactly as in
\eqref{eq:running_cost_expansion}, using the Fr\'echet differentiability
of \(g\) and the adjoint identity \eqref{eq:adjoint_identity}, and letting \(h\downarrow 0\), we obtain
\[
\begin{aligned}
    &\partial^+\rho(t,x_t)
    \bigl(\lambda,\lambda\dot x(t+)+\theta\bigr)      \\
    &\le
    p(t)\cdot v(t)
    -
    \lambda L(t,x(t),x(D(t)),\alpha(t+)),
\end{aligned}
\]
where \(\alpha(t+)\) denotes the right limit of \(\alpha\) at \(t\). Since \(v(t)=\theta\), this reads
\[
\begin{aligned}
    \partial^+\rho(t,x_t)
    \bigl(\lambda,\lambda\dot x(t+)+\theta\bigr)
    &\le
    p(t)\cdot\theta
    -
    \lambda L(t,x(t),x(D(t)),\alpha(t+))              \\
    &=
    p(t)\cdot(\lambda\dot x(t+)+\theta)
    -
    \lambda
    \bigl[
        p(t)\cdot\dot x(t+)
        +
        L(t,x(t),x(D(t)),\alpha(t+))
    \bigr].
\end{aligned}
\]
Since \(\alpha\) is continuous at \(t\), by the Pontryagin condition \eqref{eq:PMP}, we have
\[
\begin{aligned}
    p(t)\cdot\dot x(t+)
    +
    L(t,x(t),x(D(t)),\alpha(t+))&=p(t)\cdot f(t,x(t),x(D(t)),\alpha(t))+L(t,x(t),x(D(t)),\alpha(t))\\
    &=
    H(t,x(t),x(D(t)),p(t)).
\end{aligned}
\]
Therefore,
\[
\begin{aligned}
    \partial^+\rho(t,x_t)
    \bigl(\lambda,\lambda\dot x(t+)+\theta\bigr)
    \le
    -\lambda H(t,x(t),x(D(t)),p(t))
    +
    p(t)\cdot(\lambda\dot x(t+)+\theta).
\end{aligned}
\]
Since \(\lambda\ge0\) and \(\theta\in\mathbb R^n\) are arbitrary, this implies that
\[
    \bigl(-H(t,x(t),x(D(t)),p(t)),p(t)\bigr)
    \in
    \partial^{\mathrm{delay}}_+\rho(t,x_t)
\]
which yields the conclusion.
\end{proof}
\section{Semiconcavity}\label{sec:semiconcave}

In this section we prove semiconcavity properties of the value functional.
We first state the additional regularity assumptions used throughout this
section.

\begin{assumption}\label{ass:semiconcavity}
The following conditions hold.

\begin{enumerate} 
    \item \(f,L\in C^1(I\times \mathbb {R}^n \times \mathbb {R}^n \times A)\).
  
\item There exist \(\mu,\lambda_{\mathcal H}>0\) such that
\[
    (\mathcal H_a(t,z,y,p,a_1)-\mathcal H_a(t,z,y,p,a_2))\cdot (a_1-a_2)
    \geq \mu \lVert a_1-a_2\rVert_{m}^2,
\]
and
\[
    \begin{aligned}
    &\left\lVert
    \mathcal H_a(t_1,z_1,y_1,p_1,a)-\mathcal H_a(t_2,z_2,y_2,p_2,a)
    \right\rVert\\
    &\leq \lambda_{\mathcal H}\left( |t_1-t_{2}|+\lVert z_1- z_2\rVert+\|y_1-y_2\|+\|p_1-p_2\| \right)
    \end{aligned}
\]
for all \((t,z,y,p),(t_1,z_1,y_1,p_1),(t_2,z_2,y_2,p_2)\in I\times \mathbb R^n\times \mathbb R^n\times \mathbb R^n\) and \(a,a_1,a_2\in A\). Here, \(\left\Vert \cdot \right\Vert _{m}\) denotes the usual norm in \(\mathbb{R}^{m}\).

    \item For every \(r>0\), there exist constants
    \(\gamma_f(r),\gamma_L(r)>0\) such that
    \[
    \begin{aligned}
        &\left\|
        f(t,z_1,y_1,a)
        +
        f(t,z_2,y_2,a)
        -
        2f\left(
            t,\frac{z_1+z_2}{2},\frac{y_1+y_2}{2},a
        \right)
        \right\|        \\
        &\qquad\le
        \gamma_f(r)
        \bigl(
            \|z_1-z_2\|^2+\|y_1-y_2\|^2
        \bigr),
    \end{aligned}
    \]
    and
    \[
    \begin{aligned}
        &\left|
        L(t,z_1,y_1,a)
        +
        L(t,z_2,y_2,a)
        -
        2L\left(
            t,\frac{z_1+z_2}{2},\frac{y_1+y_2}{2},a
        \right)
        \right|        \\
        &\qquad\le
        \gamma_L(r)
        \bigl(
            \|z_1-z_2\|^2+\|y_1-y_2\|^2
        \bigr)
    \end{aligned}
    \]
    for all
    \(t\in I\), \(a\in A\), and
    \(z_i,y_i\in B_r\), \(i=1,2\).

    \item For every \(r>0\), there exists
    \(\gamma_\tau(r)>0\) such that
    \[
        \left|
        \tau(t,z_1)+\tau(t,z_2)
        -
        2\tau\left(t,\frac{z_1+z_2}{2}\right)
        \right|
        \le
        \gamma_\tau(r)\|z_1-z_2\|^2
    \]
    for all \(t\in I\) and \(z_1,z_2\in B_{r}\).
    \item For any $(t,w)\in \mathbb{G}$ and $\alpha \in \mathcal{A}_{t,T}$, letting $x=x(\cdot; t,w,\alpha)$, we assume that the terminal cost $g$ is Fr\'echet differentiable at $(x(T),x_{T})$ with derivative $(g_{z}(x(T),x_{T}),g_{w}(x(T),x_{T}))$. We further assume that, for $R>0$, there exists $\gamma _{g}(r)>0$ such that $\left\Vert g_{w}(x(T),x_{T}) \right\Vert _{\infty}\leq \gamma _{g}(r)$ for all $t\in I$, $w\in V(r)$ and $\alpha \in \mathcal{A}_{t,T}$.
\end{enumerate}
\end{assumption}

For the second-order estimates developed below, we need to select optimal controls with Lipschitz representatives and uniformly controlled Lipschitz constants. This regularity is essential for controlling the second-order variations of the trajectories and, in particular, of the composite delayed states. Using the Pontryagin minimum principle from the previous section, we now establish the existence of optimal controls with the required regularity.

\begin{lemma}\label{lem:Lip_control}
There exists an optimal control for any $(t,w)\in \mathbb{G}$. In addition, the optimal control admits a Lipschitz representative. Moreover, for $t\in I$, $r>0$ and $w\in V(r)$, there exists $C(r)>0$ such that $\operatorname{Lip}(\alpha )\leq C(r)$ for all $\alpha$ being the optimal control of $(t,w)$.
\end{lemma}
\begin{proof}
Let \(\alpha_k\in \mathcal A_{t,T}\) be a minimizing sequence for \(\rho(t,w)\) and \(x_k\) be the corresponding trajectories. Up to a subsequence, \(x_k\) converges uniformly to a Lipschitz mapping \(\bar{x}\) by the Arzel\`a-Ascoli theorem. Then, Filippov's lemma implies that there exists \(\bar{\alpha}\in \mathcal A_{t,T}\) such that \(\bar{x}\) is the trajectory associated with \(\bar \alpha\). By the Lipschitz continuity of \(L\), \(g\), \(\tau\) and the fact that \(\operatorname{Lip}(x_k)\) is uniformly bounded, we deduce that \((\bar x, \bar \alpha)\) is an optimal pair.

Define 
\[
\kappa(Y)=\argmin_{a\in A} \mathcal{H} (Y,a)
\]
for \(Y:=(t,z,y,p)\in I\times \mathbb R^n\times \mathbb R^n\times \mathbb R^n\). By the standard sensitivity estimate for strongly monotone variational inequalities, \(\kappa:  I\times \mathbb R^n\times \mathbb R^n\times \mathbb R^n \to A\) is single-valued and Lipschitz with $\operatorname{Lip}(\kappa) \leq \lambda _{\mathcal{H} }/\mu$. Applying \Cref{thm:PMP}, we find the optimal control $\bar{\alpha}$ satisfies
\[
\bar{\alpha } (s)=\kappa(s,x(s),x(D(s)),p(s)), \quad \text{a.e. } s\in [t,T].
\]
Since \(x, D, p,\kappa\) are all Lipschitz, we obtain a Lipschitz representative for $\bar{\alpha }$. 

If $t\in I$ and $w\in V(r)$, since $\left\Vert g_{w}(x(T),x_{T}) \right\Vert _{\infty}$ are uniformly bounded by $\gamma (r)$, we deduce from the equation of $p$ that $\operatorname{Lip}(p)\leq C(r)$. Thus $\operatorname{Lip}(\alpha )\leq C(r)$.
\end{proof}

For \(r>0\), let \(W(r)\) denote the set of all
piecewise \(C^{1,1}\) functions \(w\in V(r)\) satisfying
\[
    \operatorname{Lip}\!\left(
        \dot w\big|_{(\theta_{j-1},\theta_j)}
    \right)\le r
\]
on every subinterval \((\theta_{j-1},\theta_j)\) on which \(w\) is
\(C^{1,1}\). Furthermore, for \(N\in\mathbb N\), let \(W(r,N)\) denote the set of all functions \(w\in W(r)\) having at most \(N\) nonsmooth points.
\begin{lemma}\label{lem:semiconcavity_trajectories}
Let \(t_0\le t_-<t_+\le T\) and \(w_i\in W(r,N)\) with \(z_{i}=w_{i}(0)\) for $N\in \mathbb{N}$ and \(r>0\), \(i=1,2\). Let \(\alpha\in \operatorname{Lip}\left( [t_{-},t_{+}]; A \right) \) and
    \[
        x_i(\cdot)
        :=
        x(\cdot\,|\,t_-,w_i,\alpha),
        \qquad i=1,2.
    \]

\begin{enumerate}
    \item 
There exists
\(C(r)>0\) such that
\begin{equation}\label{eq:trajectory_lipschitz_initial}
        \|x_1(s)-x_2(s)\|
        \le
        C(r)
        \bigl(
            \|z_1-z_2\|+\|w_1-w_2\|_1
        \bigr),
        \qquad
        s\in[t_-,t_+].
    \end{equation}

    \item 
   Let
    \[
        w_3:=\frac{w_1+w_2}{2},
        \qquad
        z_3:=w_3(0)=\frac{z_1+z_2}{2},
    \]
    and set
    \[
        x_3(\cdot):=x(\cdot\,|\,t_-,w_3,\alpha).
    \]
Then there exists
\(C=C(r,N,\operatorname{Lip}(\alpha ))>0\) such that
    \begin{equation}\label{eq:trajectory_semiconcavity}
        \|x_1(s)+x_2(s)-2x_3(s)\|
        \le
        C
        \|w_1-w_2\|_{H^1}^2
    \end{equation}
    for all \(s\in[t_-,t_+]\). In particular, \eqref{eq:trajectory_semiconcavity} holds for \(\alpha\in \operatorname{Lip}\left((t_-,t_+];A\right)\).
\end{enumerate}
\end{lemma}

\begin{proof}
The first estimate \eqref{eq:trajectory_lipschitz_initial} follows directly from \Cref{lem:continuous_dependence}, applied on the time interval $[t_-,t_+]$.

We next show \eqref{eq:trajectory_semiconcavity}. For \(i=1,2,3\), set
\[
    D_i(s):=s-\tau(s,x_i(s)),\quad s\in [t_-,t_+].
\]
Define
\[
    q(s)=x_1(s)+x_2(s)-2x_3(s).
\]
Using \Cref{ass:semiconcavity}, we obtain
\[
\begin{aligned}
    \|q(s)\|
    &\le
    C(r)\bigg[
    \int_{t_-}^s
    \bigl(
        \|x_1(\xi)-x_2(\xi)\|^2
        +
        \|x_1(D_1(\xi))-x_2(D_2(\xi))\|^2
    \bigr)\,d\xi       \\
    &\quad
    +
    \int_{t_-}^s \|q(\xi)\|\,d\xi+\int_{t_-}^s A(\xi)\,d\xi
    \bigg]
\end{aligned}
\]
for \(s\in[t_-,t_+]\), where
\[
A(\xi)=
    \|x_1(D_1(\xi))+x_2(D_2(\xi))-2x_3(D_3(\xi))\|.
\]

By the monotonicity assumption on the delay and the change of variables \(\eta=D_1(\xi)\) as in the proof of \Cref{lem:continuous_dependence}, we have
\begin{equation}\label{eq:D1D2}
\begin{aligned}
    &\int_{t_-}^s
    \|x_1(D_1(\xi))-x_2(D_2(\xi))\|^2\,d\xi              \\
    &\qquad\le
    C(r)
    \left(
        \|w_1-w_2\|_2^2
        +
        \int_{t_-}^s
        \|x_1(\xi)-x_2(\xi)\|^2\,d\xi
    \right).
\end{aligned}
\end{equation}

To deal with the integral of \(A(\xi)\), we set
\[
    m(\xi):=\frac{D_1(\xi)+D_2(\xi)}{2},\quad \xi\in [t_-,t_+].
\]
Denote 
\[
\delta:= \|w_1-w_2\|_{H^1}>0.
\]
By the Lipschitz property of \(\tau\), \eqref{eq:trajectory_lipschitz_initial} yields
\begin{equation}\label{eq:D12_m}
|D_1(\xi)-m(\xi)|=|D_2(\xi)-m(\xi)|\leq C\|x_1(\xi)-x_2(\xi)\|\leq C_0\delta
\end{equation}
for all \(\xi\in [t_-,t_+]\). 

We know that \(x_i:[t_{-}-\bar{\tau} ,t_{+}] \to \mathbb R^n\) is piecewise \(C^{1,1}\), \(i=1,2\). We collect all the nonsmooth points of \(w_{1}\) and \(w_{2}\) and denote them by \(t_{1}\leq t_{2}\leq \cdots \leq t_{k}=t_{-}\) for some \(k\in \mathbb{N}\).

Set \(I_0=[t_{-},t_{+}]\) and \(I_{m}=[m(t_{-}),m(t_{+})]\). Then \(I_{m}\neq \emptyset\) since \(m\) is strictly increasing. Let \(I_{i}=I_{m}\cap [t_{i}-C_{0}\delta ,t_{i}+C_{0}\delta ]\), \(i=1,2,\ldots,k\). For any fixed \(\xi \in I_0\), we note that \(D_{1}(\xi ),D_{2}(\xi ),m(\xi )\) must lie on the same side of \(t_{i}\) when \(m(\xi )\in I_{m}\setminus I_{i}\) by \eqref{eq:D12_m}. In particular, \(I_{i}=\emptyset\) implies that, for every fixed \(\xi \in I_0\), the points \(D_{1}(\xi ),D_{2}(\xi ),m(\xi )\) must lie on the same side of \(t_{i}\). We denote \(J_{i}=m ^{-1} I_{i}\). Since \(\dot{m}\) is bounded below by a positive constant on \([t_-,t_+]\), we have \(|J_i|< C_1 \delta\) for some \(C_1>0\). 

\textcolor{red}{Since each trajectory \( x_i\colon [t_{-}-\bar{\tau},t_{+}] \to \mathbb{R}^n\), with \( i=1,2, \) is piecewise \(C^{1,1}\), its possible nonsmooth points on \([t_{-}-\bar{\tau},t_{-}]\) are inherited from the corresponding initial history. Let \[ t_1 \leq t_2 \leq \cdots \leq t_k=t_{-} \] be the collection of all nonsmooth points of \(w_1\) and \(w_2\), including the junction point \(t_{-}\). Set \[ I_0=[t_{-},t_{+}] \qquad\text{and}\qquad I_m=m(I_0)=[m(t_{-}),m(t_{+})]. \] Since \(m\) is strictly increasing, \(I_m\) is a nonempty interval. For each \(i=1,\ldots,k\), define \[ I_i:=I_m\cap [t_i-C_0\delta,t_i+C_0\delta] \qquad\text{and}\qquad J_i:=m^{-1}(I_i). \] Thus, \(J_i\) consists of those times \(\xi\in I_0\) for which \(m(\xi)\) lies within distance \(C_0\delta\) of the nonsmooth point \(t_i\). If \(\xi\in I_0\setminus J_i\), then \[ |m(\xi)-t_i|>C_0\delta. \] Together with \eqref{eq:D12_m}, this implies that \(D_1(\xi)\), \(D_2(\xi)\), and \(m(\xi)\) all lie on the same side of \(t_i\). Consequently, if \[ \xi\in I_0\setminus\bigcup_{i=1}^k J_i, \] then these three points belong to a common subinterval on which both \(x_1\) and \(x_2\) are \(C^{1,1}\). Moreover, if \(I_i=\emptyset\), then \(J_i=\emptyset\), and the preceding conclusion with respect to \(t_i\) holds for every \(\xi\in I_0\). Finally, since \(\dot m\) is bounded below by a positive constant, there exists \(c_m>0\) such that \[ \dot m(\xi)\geq c_m \qquad\text{for a.e. }\xi\in I_0. \] Hence, \[ |J_i| \leq \frac{|I_i|}{c_m} \leq \frac{2C_0}{c_m}\,\delta. \] In particular, there exists \(C_1>0\), independent of \(i\), such that \[ |J_i|\leq C_1\delta, \qquad i=1,\ldots,k. \]}

We claim that
\begin{equation}\label{int_A}
\int_{t_-}^s A(\xi)\,d\xi\leq C(r,N,\operatorname{Lip}(\alpha))(\int_{t_-}^{s}\|q(\xi)\|\,d\xi+ \delta^2)
\end{equation}
for all \(s\in I_0\). By \eqref{eq:D1D2}, \eqref{eq:trajectory_lipschitz_initial}, Gr\"onwall's inequality, and the Sobolev embedding of \(H^1\), we then obtain \eqref{eq:trajectory_semiconcavity}.

Now we prove the claim. Note that, for fixed \(\xi \in  I_0\setminus \bigcup_{i=1}^{k}J_{i}\), the points \(D_{1}(\xi ),D_{2}(\xi ),m(\xi )\) must lie in some common \(C^{1,1}\) subinterval of \(x_{1}\) and \(x_{2}\). By Taylor's formula, we have
\[
\begin{aligned}
&x_1(D_1(\xi))=x_1(m(\xi))+\dot x_{1}(m(\xi ))(D_{1}(\xi )-m(\xi))+R_{1}(\xi ),\\
&x_{2}(D_{2}(\xi))=x_2(m(\xi))+\dot x _{2}(m(\xi ))(D_{2}(\xi )-m(\xi))+R_{2}(\xi ).
\end{aligned}
\]
If \(m(\xi)<t_{-}\), then \(\dot x_{1}(m(\xi ))=\dot w_1(m(\xi )-t_{-})\) and \(|R_{i}(\xi )|\leq C(r)\|x_{1}(\xi )-x_{2}(\xi )\|^2\) by \eqref{eq:D12_m}. Thus we obtain
\[
  A(\xi)\le C(r)\left(\|q(m(\xi))\|+\left\Vert q(\xi ) \right\Vert +\|\dot{w}_1(m(\xi)-t_{-})-\dot{w}_{2}(m(\xi)-t_{-})\|^2+\|x_1(\xi)-x_2(\xi)\|^2\right).
\]
Here we used the semiconcavity of \(\tau\), which gives
\[
    |m(\xi)-D_3(\xi)|
    \le
    C(r)
    \bigl(
        \|x_1(\xi)-x_2(\xi)\|^2+\|q(\xi)\|
    \bigr).
\]
Moreover,
\[
\begin{aligned}
    &\int_{t_-}^s
    \|\dot w_1(m(\xi)-t_-)-\dot w_2(m(\xi)-t_-)\|^2\,d\xi\le 
    C(r)\delta ^{2}.
\end{aligned}
\]
If \(m(\xi)>t_{-}\), the expansion argument then yields
\[
A(\xi)\le C(r,\operatorname{Lip}(\alpha))\left(\|q(m(\xi))\|+\left\Vert q(\xi ) \right\Vert +\|\dot{x}_1(m(\xi))-\dot{x}_{2}(m(\xi))\|^2+\|x_1(\xi)-x_2(\xi)\|^2\right).
\]
Moreover, we have 
\[
\begin{aligned}
    &\int_{t_-}^s
    \|\dot x_1(\xi)-\dot x_2(\xi)\|^2\,d\xi\\
    &= \int_{t_-}^s
    \|f(\xi,x_1(\xi),x_1(D_1(\xi)),\alpha(\xi))-f(\xi,x_2(\xi),x_2(D_2(\xi)),\alpha(\xi))\|^2\,d\xi\\
    &\le
    C(r)
     \int_{t_-}^s
    \bigl(
        \|x_1(\xi)-x_2(\xi)\|^2
        +
        \|x_1(D_1(\xi))-x_2(D_2(\xi))\|^2
    \bigr)\,d\xi\\
    &\le C(r)\delta^2
\end{aligned}
\]
for \(s\in I_0\).

Next we deal with \(J_{i}\). From \eqref{eq:trajectory_lipschitz_initial}, we know that
\[
\sup_{\xi\in [t_-,t_+]}\|x_1(\xi)-x_3(\xi)\|\leq C(r) \delta
\]
and
\[
\sup_{\xi\in [t_-,t_+]}\|x_2(\xi)-x_3(\xi)\|\leq C(r) \delta.
\]
Thus we have
\[
\begin{aligned}
\int_{J_{i}} A(\xi)\,d\xi \leq& \int_{J_{i}} \|x_1(D_1(\xi))-x_1(D_3(\xi))\|+\|x_2(D_2(\xi))-x_2(D_3(\xi))\|\\
&+\|x_1(D_3(\xi))-x_3(D_3(\xi))\|+\|x_2(D_3(\xi))-x_3(D_3(\xi))\| \,d\xi\\
\leq & C(r)\delta^2.
\end{aligned}
\]
Hence 
\[
  \int_{\bigcup_{i=1}^{k} J_{i}} A(\xi)\,d\xi\leq C(r,N)\delta ^{2}.
\]

Using the above estimates and changing variables \(m(\xi )\mapsto \xi \), we obtain the claim.
\end{proof}
\begin{remark}\label{rmk:zero_delay}
If \(D_1(t_{-})=D_2(t_-)=t_-\), then \(D_1(\xi),D_2(\xi),m(\xi)\geq t_{-}\) for all \(\xi\geq t_{-}\). In this case, \eqref{eq:trajectory_semiconcavity} can be replaced by
\[
\|x_1(s)+x_2(s)-2x_3(s)\|\le C(r)(\|z_1-z_2\|^2+\|w_1-w_2\|_2^2).
\]
\end{remark}

\begin{theorem}\label{thm:space_semiconcavity_value}
Assume that \(g\) is locally semiconcave, i.e., for \(r>0\), there exists \(C(r)\) such that
\[
    g(z+z_0,w+w_0)
    +
    g(z-z_0,w-w_0)
    -
    2g(z,w)                                 
    \le
    C(r)
      \|w_0\|_{H^1}^2
\]
for all $(z,w),(z_{0},w_{0})\in \mathcal{X}$ with \(w,w_{0}\in W(r)\). Then the value functional
\(\rho\) is locally semiconcave with respect to the history variable. That is, there exists \(C(r,N)\) such that
\[
    \rho (t,w+w_0)
    +
    \rho (t,w-w_0)
    -
    2\rho (t,w)                                 
    \le
    C(r,N)
      \|w_0\|_{H^1}^2,
\]
whenever
\[
    t\in I,\quad w,w_0 \in W(r,N).
\]
\end{theorem}

\begin{proof}
Fix \(r>0\). By \Cref{lem:trajectory_bounds}, all trajectories starting from histories in \(W(r)\) are bounded on \([t,T]\) for all admissible controls.

Let \(t\in I\) and let \(w,w_{0}\in W(r,N)\) with \(z=w(0)\) and \(z_0=w_0(0)\) for \(r>0\). Let \(\alpha\in \mathcal{A}_{t,T}\) be an optimal control for \((t,w)\). Then \Cref{lem:Lip_control} implies that $\operatorname{Lip}(\alpha )\leq C(r)$.
Let
\[
    x:=x(\cdot\,|\,t,w,\alpha),
    \qquad
    x_+:=x(\cdot\,|\,t,w+w_0,\alpha),
    \qquad
    x_-:=x(\cdot\,|\,t,w-w_0,\alpha).
\]
By \Cref{lem:semiconcavity_trajectories} and the Sobolev embedding, there exists \(C(r)>0\) such that 
\begin{equation}\label{eq:xpm_semiconcavity}
    \|x_+(s)-x_-(s)\|
    \le
    C(r)\|w_0\|_{H^{1}}
\end{equation}
and
\begin{equation}\label{eq:xpm_second_order}
    \|x_+(s)+x_-(s)-2x(s)\|
    \le
    C(r)\|w_0\|_{H^1}^2
\end{equation}
for all \(s\in[t,T]\).

Set
\[
D_{\pm}(s)=s-\tau(s,x_\pm(s)), \quad D(s)=s-\tau(s,x(s)) \qquad s\in [t,T].
\]
Using the definition of \(\rho\) and the optimality of
\(\alpha\), we obtain
\[
\begin{aligned}
    &\rho(t,w+w_0)+\rho(t,w-w_0)-2\rho(t,w)        \\
    &\le
    g(x_+(T),(x_+)_T)
    +
    g(x_-(T),(x_-)_T)
    -
    2g(x(T),x_T)                                    \\
    &\quad
    +
    \int_t^T
    \Bigl[
        L(s,x_+(s),x_+(D_+(s)),\alpha(s))           \\
    &\qquad\qquad
        +
        L(s,x_-(s),x_-(D_-(s)),\alpha(s))           \\
    &\qquad\qquad
        -
        2L(s,x(s),x(D(s)),\alpha(s))
    \Bigr]\,ds.
\end{aligned}
\]

By the semiconcavity and the Lipschitz
continuity of \(g\), \eqref{eq:xpm_semiconcavity} and \eqref{eq:xpm_second_order}, we estimate
\[
\begin{aligned}
    &g(x_+(T),(x_+)_T)
    +
    g(x_-(T),(x_-)_T)
    -
    2g(x(T),x_T)                                    \\
    &\le
    C(r)
    \bigg[
        \|(x_+)_T-(x_-)_T\|_{H^{1}}^2
                                                    +
        \|x_+(T)+x_-(T)-2x(T)\|
        +
        \|(x_+)_T+(x_-)_T-2x_T\|_1
   \bigg].              
\end{aligned}
\]
As in the proof of \Cref{lem:semiconcavity_trajectories}, along with \eqref{eq:xpm_semiconcavity}, we deduce
\[
  \|(x_+)_T-(x_-)_T\|_{H^{1}}^2\leq C(r)\left\Vert w_{0} \right\Vert _{H^{1}}^2.
\]
Then \eqref{eq:xpm_second_order} yields that
\[
  g(x_+(T),(x_+)_T)
    +
    g(x_-(T),(x_-)_T)
    -
    2g(x(T),x_T) \leq C(r)\left\Vert w_{0} \right\Vert _{H^{1}}^2.
\]

Using the semiconcavity of
\(L\), as in the proof of \Cref{lem:semiconcavity_trajectories}, we obtain
\[
\begin{aligned}
    &\int_t^T
    \Bigl|
        L(s,x_+(s),x_+(D_+(s)),\alpha(s))
        +
        L(s,x_-(s),x_-(D_-(s)),\alpha(s))  \\
    &\qquad\qquad
        -
        2L(s,x(s),x(D(s)),\alpha(s))
    \Bigr|\,ds                                    \\
    &\quad\le
    C(r)
       \|w_0\|_{H^1}^2.
\end{aligned}
\]
Thus
\[
    \rho(t,w+w_0)+\rho(t,w-w_0)-2\rho(t,w)
    \le
    C(r)\|w_0\|_{H^1}^2.
\]
\end{proof}

We conclude this section with a joint semiconcavity estimate. To state it, for \(r>0\), define \[ W'(r):= W(r)\cap C^{1,1}\bigl([-\bar{\tau},0];\mathbb{R}^n\bigr). \] We also introduce the solution manifold associated with \eqref{eq:FDE}. For \(t\in I\) and \(a\in A\), let \[ X_{t,a}:= \left\{ w\in C^{1,1}\bigl([-\bar{\tau},0];\mathbb{R}^n\bigr) : \dot{w}(0) = f\bigl(t,w(0),w(-\tau(t,w(0))),a\bigr) \right\}. \] Thus, \(X_{t,a}\) consists of the histories satisfying the compatibility condition at the junction between the prescribed history and the forward solution of \eqref{eq:FDE}. We refer to \(X_{t,a}\) as the \(C^{1,1}\) solution manifold associated with \(t\) and \(a\). Since the delay is discrete, a \(C^{1,1}\) history satisfying this compatibility condition can be constructed by interpolation. Hence, \[ X_{t,a}\neq\emptyset \qquad \text{for every }(t,a)\in I\times A. \]

\begin{theorem}\label{thm:time_semiconcavity_value}
Assume that \(g\) is locally semiconcave. Let \(r>0\) and \(h>0\). Then there exists \(C(r)>0\) such that
\[
\begin{aligned}
    &\rho(t+h,w+w_0)
    +
    \rho(t-h,w-w_0)
    -
    2\rho(t,w)                                      \\
    &\qquad\le
    C(r)
    \bigl(
        h^2+\|w_0\|_{H^1}^2
    \bigr)
\end{aligned}
\]
whenever
\[
    t_0\le t-h< t+h\le T,   \qquad
   w,w_0\in W'(r), \qquad w-w_{0}\in X_{t,\alpha(t)},
\]
where \(\alpha \) is an optimal control for \((t,w)\).
\end{theorem}

\begin{proof}
Let \(z=w(0)\) and \(z_0=w_0(0)\). Set
\[
    x(s):=x(s\,|\,t,w,\alpha),
    \qquad
    s\in[t,t+h].
\]
Define the rescaled control
\[
    \bar\alpha(s)
    :=
    \alpha\left(\frac{t+h+s}{2}\right),
    \qquad
    s\in [t-h,t+h],
\]
and let
\[
    \bar x(s)
    :=
    x(s\,|\,t-h,w-w_0,\bar\alpha),
    \qquad
    s\in[t-h,t+h].
\]
Put
\[
    \hat z:=x(t+h),
    \qquad
    \bar z:=\bar x(t+h).
\]
By the local boundedness of \(f\), there exists
\(M_f>0\) such that, for all
\(s_1\in[t-h,t+h]\) and \(s_2\in[t,t+h]\),
\begin{equation}\label{eq:barx_x_difference}
    \|\bar x(s_1)-x(s_2)\|
    \le
    \|z_0\|+3M_f h.
\end{equation}
Also,
$$
\|z+z_0-\bar{z}\|=\|2z_0-\int_{t-h}^{t+h}f(s,\bar{x}(s),\bar{x}(s-\tau(s,\bar x(s))),\bar{\alpha}(s))ds\|\leq 2\|z_0\|+2M_f h.
$$

Let
\[
    D(s):=s-\tau(s,x(s)),
    \qquad
    \bar D(s):=s-\tau(s,\bar x(s)).
\]
By Lipschitz properties of $D,w,w_{0},x$ and $\bar{x}$, it can be checked that
\[
    \int_{D(t)}^{D(t+h)}
    \|\bar x(s)-x(s)\|\,ds
    \le
    C(r)\bigl(h^{2}+\|w_0\|_\infty^{2}\bigr).
\]
Using the change of variable
\[
    \eta(s):=2s-t-h,
    \qquad
    s\in[t,t+h],
\]
the local Lipschitz properties of \(f\) and \(\tau\),
and \eqref{eq:barx_x_difference}, we also obtain
\begin{equation}\label{eq:zbar_second_order}
    \|z+z_0+\bar z-2\hat z\|
    \le
    C(r)
        \bigl(
        h^{2}+\|w_0\|_\infty^{2}
    \bigr).
\end{equation}

By the dynamic programming principle,
\[
\begin{aligned}
    &\rho(t+h,w+w_0)
    +
    \rho(t-h,w-w_0)
    -
    2\rho(t,w)                                      \\
    &\le
    \rho(t+h,w+w_0)
    +
    \rho(t+h,\bar x_{t+h})
    -
    2\rho(t+h,x_{t+h})                              \\
    &\quad
    +
    \int_{t-h}^{t+h}
    L(s,\bar x(s),\bar x(\bar D(s)),\bar\alpha(s))\,ds
    -
    2\int_t^{t+h}
    L(s,x(s),x(D(s)),\alpha(s))\,ds.
\end{aligned}
\]
The first three terms on the right-hand side are estimated by \Cref{thm:space_semiconcavity_value} and the Lipschitz continuity of
\(\rho\):
\[
\begin{aligned}
    &\rho(t+h,w+w_0)
    +
    \rho(t+h,\bar x_{t+h})
    -
    2\rho(t+h,x_{t+h})                              \\
    &\le
    C(r)
        \|w+w_0-\bar x_{t+h}\|_{H^{1}}^2                                    \\
    &\quad
    +
    C(r)
    \bigl(
        \|z+z_0+\bar z-2\hat z\|
        +
        \|w+w_0+\bar x_{t+h}-2x_{t+h}\|_1
    \bigr).
\end{aligned}
\]
We now estimate the history terms. By Lipschitz properties of $w$ and $\bar{x} $, we find
\[
\begin{aligned}
    \|w+w_0-\bar x_{t+h}\|_2^2
    \le
    C(r)
    \bigl(
        h^2+\|w_0\|_\infty^2
    \bigr).
\end{aligned}
\]
Since \(w-w_{0}\in X_{t,\alpha (t)}\), we deduce from the \(C^{1}\)-regularity of \(f\) and \(\tau\) that
\[
\begin{aligned}
  &\left\Vert \dot{\bar{x}}(t-h)-(\dot{w}-\dot{w}_{0})(0)  \right\Vert\\
  &=\left\Vert f(t-h,z-z_{0},(w-w_{0})(-\tau(t-h,z-z_{0})),\alpha (t))-f(t,z-z_{0},(w-w_{0})(-\tau(t,z-z_{0})),\alpha (t)) \right\Vert\\
  &\leq C(r)h,
\end{aligned}
\]
where \(\dot {\bar{x}}(t-h)\) denotes the right derivative. Since \(\operatorname{Lip}(\dot w_0)\le r\), we insert the terms \(\dot{\bar{x}}(t-h)\) and \((\dot{w}-\dot{w}_{0})(0)\) to estimate
\[
\begin{aligned}
  &\int_{-\bar{\tau}}^{0} \left\Vert \dot{w}(s)+\dot{w}_{0}(s)-\dot{\bar{x} }(t+h+s)   \right\Vert^{2} \, ds\\
  &=  \int_{-\bar{\tau}}^{-2h} \left\Vert \dot{w}(s)+\dot{w}_{0}(s)-(\dot{w}-\dot{w_{0}}) (2h+s)   \right\Vert^{2} \, ds+\int_{-2h}^{0} \left\Vert \dot{w}(s)+\dot{w}_{0}(s)-\dot{\bar{x} }(t+h+s)   \right\Vert^{2} \, ds\\
  &\leq C(r)(h^{2}+\|w_0\|_{H^1}^2),
\end{aligned}
\]
where we use the Lipschitz continuity of \(\dot{\bar{x}}\) on \([t-h,t+h]\).
Using the semiconcavity of \(w\), the Lipschitz continuity of
\(w_0\), and \eqref{eq:barx_x_difference}, we obtain
\[
\begin{aligned}
    \|w+w_0+\bar x_{t+h}-2x_{t+h}\|_1
    \le
    C(r) 
        \bigl(
        h^{2}+\|w_0\|_{H^1}^{2}
    \bigr).
\end{aligned}
\]
Finally, by the local Lipschitz continuity of \(L\) and the same argument as in the estimate \eqref{eq:zbar_second_order}, we derive
\[
\begin{aligned}
    &\int_{t-h}^{t+h}
    L(s,\bar x(s),\bar x(\bar D(s)),\bar\alpha(s))\,ds
    -
    2\int_t^{t+h}
    L(s,x(s),x(D(s)),\alpha(s))\,ds                 \\
    &\qquad\le
    C(r) 
    \bigl(
        h^{2}+\|w_0\|_\infty^{2}
    \bigr).
\end{aligned}
\]
Combining the previous inequalities yields
\[
\begin{aligned}
    &\rho(t+h,w+w_0)
    +
    \rho(t-h,w-w_0)
    -
    2\rho(t,w)                                      \\
    &\qquad\le
    C(r)
    \bigl(
        h^2+\|w_0\|_{H^1}^2
    \bigr).
\end{aligned}
\]
\end{proof}




\begin{thebibliography}{10} 

\bibitem{MR231007} 
H.~T. Banks. 
\newblock Necessary conditions for control problems with variable time lags. 
\newblock {\em SIAM J. Control}, 6:9--47, 1968. 

\bibitem{MR3702857} 
A.~Boccia and R.~B. Vinter. 
\newblock The maximum principle for optimal control problems with time delays. 
\newblock {\em SIAM J. Control Optim.}, 55(5):2905--2935, 2017. 


\bibitem{MR1014944} Martin Brokate and Fritz Colonius. 
\newblock Linearizing equations with state-dependent delays. 
\newblock {\em Appl. Math. Optim.}, 21(1):45--52, 1990. 

\bibitem{MR2041617} Piermarco Cannarsa and Carlo Sinestrari. 
\newblock {\em Semiconcave functions, {H}amilton-{J}acobi equations, and optimal control}, volume~58 of {\em Progress in Nonlinear Differential Equations and their Applications}. \newblock Birkh\"auser Boston, Inc., Boston, MA, 2004. 


\bibitem{MR4893227} 
Elisa Continelli and Cristina Pignotti.
\newblock Semiconcavity for the minimum time problem in presence of time delay effects. \newblock {\em Math. Control Relat. Fields}, 15(3):1020--1048, 2025. 

\bibitem{MR732102} M.~G. Crandall, L.~C. Evans, and P.-L. Lions.
\newblock Some properties of viscosity solutions of {H}amilton-{J}acobi equations.
\newblock {\em Trans. Amer. Math. Soc.}, 282(2):487--502, 1984. 

\bibitem{MR690039} 
Michael~G. Crandall and Pierre-Louis Lions. 
\newblock Viscosity solutions of {H}amilton-{J}acobi equations. 
\newblock {\em Trans. Amer. Math. Soc.}, 277(1):1--42, 1983. 


\bibitem{MR794776} Michael~G. Crandall and Pierre-Louis Lions. 
\newblock Hamilton-{J}acobi equations in infinite dimensions. {I}. {U}niqueness of viscosity solutions. 
\newblock {\em J. Funct. Anal.}, 62(3):379--396, 1985. 


\bibitem{PhysRevE.111.035313} G.~O. Danilenko, A.~V. Kovalev, D.~S. Citrin, A.~Locquet, D.~Rontani, and E.~A. Viktorov. 
\newblock Reservoir computing with state-dependent time delay. 
\newblock {\em Phys. Rev. E}, 111:035313, March 2025. 

\bibitem{MR3369216} 
Odo Diekmann and Karol\'ina Korvasov\'a. 
\newblock Linearization of solution operators for state-dependent delay equations: a simple example. 
\newblock {\em Discrete Contin. Dyn. Syst.}, 36(1):137--149, 2016. 


\bibitem{MR150421} Rodney~D. Driver. 
\newblock Existence theory for a delay-differential system. 
\newblock {\em Contributions to Differential Equations}, 1:317--336, 1963. 

\bibitem{MR4785300}
 Johanna Frohberg and Marcus Waurick. 
\newblock State-dependent delay differential equations on {$H^1$}. 
\newblock {\em J. Differential Equations}, 410:737--771, 2024. 


\bibitem{MR247556} 
A.~Halanay. 
\newblock Optimal controls for systems with time lag. 
\newblock {\em SIAM J. Control}, 6:215--234, 1968. 


\bibitem{MR2457636} 
Ferenc Hartung, Tibor Krisztin, Hans-Otto Walther, and Jianhong Wu. 
\newblock Functional differential equations with state-dependent delays: theory and applications. 
\newblock In {\em Handbook of differential equations: ordinary differential equations. {V}ol. {III}}, Handb. Differ. Equ., pages 435--545. Elsevier/North-Holland, Amsterdam, 2006. 

\bibitem{RevModPhys.73.1067} Dirk Helbing. 
\newblock Traffic and related self-driven many-particle systems. 
\newblock {\em Rev. Mod. Phys.}, 73:1067--1141, December 2001. 


\bibitem{MR1783365} A.~V. Kim. 
\newblock {\em Functional differential equations}, volume~479 of {\em Mathematics and its Applications}. 
\newblock Kluwer Academic Publishers, Dordrecht, 1999. 
\newblock Application of $i$-smooth calculus. 


\bibitem{MR2729685} 
N.~Yu. Lukoyanov. 
\newblock On optimality conditions for the guaranteed result in control problems for time-delay systems. 
\newblock {\em Proc. Steklov Inst. Math.}, 268:S175--S187, 2010. 

\bibitem{MAHAFFY1998135} 
Joseph~M. Mahaffy, Jacques B\'elair, and Michael~C. Mackey. \newblock Hematopoietic model with moving boundary condition and state-dependent delay: Applications in erythropoiesis. 
\newblock {\em Journal of Theoretical Biology}, 190(2):135--146, 1998. 


\bibitem{martinovich_introducing_2025} 
Krist\'of Martinovich, G\'abor St\'ep\'an, D\'aniel Bachrathy, and \'Ad\'am~K. Kiss. 
\newblock Introducing state-dependent delay in the car-following model. 
\newblock {\em Nonlinear Dynamics}, 113(17):22923--22941, September 2025.


 \bibitem{pmlr-v255-monsel24a} 
 Thibault Monsel, Onofrio Semeraro, Lionel Mathelin, and Guillaume Charpiat.
 \newblock Time and state dependent neural delay differential equations. 
 \newblock In Cec{\'i}lia Coelho, Bernd Zimmering, M.~Fernanda~P. Costa, Lu{\'i}s~L. Ferr{\'a}s, and Oliver Niggemann, editors, {\em Proceedings of the 1st ECAI Workshop on ``Machine Learning Meets Differential Equations: From Theory to Applications''}, volume~255 of {\em Proceedings of Machine Learning Research}, pages 1--20. PMLR, October 2024. 
 
 \bibitem{MR4163474} 
 Anton Plaksin. 
 \newblock Minimax and viscosity solutions of {H}amilton-{J}acobi-{B}ellman equations for time-delay systems. 
 \newblock {\em J. Optim. Theory Appl.}, 187(1):22--42, 2020. 
 
 
 \bibitem{MR4264642} 
 Anton Plaksin. 
 \newblock Viscosity solutions of {H}amilton-{J}acobi-{B}ellman-{I}saacs equations for time-delay systems. 
 \newblock {\em SIAM J. Control Optim.}, 59(3):1951--1972, 2021. 
 
 
 \bibitem{MR745789} 
 A.~I. Subbotin. 
 \newblock Generalization of the main equation of differential game theory. 
 \newblock {\em J. Optim. Theory Appl.}, 43(1):103--133, 1984. 
 
 
 \bibitem{MR1320507} Andre\u{\i}~I. Subbotin. 
 \newblock {\em Generalized solutions of first-order {PDE}s}. 
 \newblock Systems \& Control: Foundations \& Applications. Birkh\"auser Boston, Inc., Boston, MA, 1995. 
 \newblock The dynamical optimization perspective. Translated from the Russian. 
 
 
 \bibitem{MR477959} 
 K.~L. Teo and E.~J. Moore. 
 \newblock Necessary conditions for optimality for control problems with time delays appearing in both state and control variables. 
 \newblock {\em J. Optim. Theory Appl.}, 23(3):413--428, 1977. 
 
 
 
 \bibitem{MR3705373} 
 R.~B. Vinter. 
 \newblock State constrained optimal control problems with time delays. 
 \newblock {\em J. Math. Anal. Appl.}, 457(2):1696--1712, 2018. 
 
 
 \bibitem{MR2019242} Hans-Otto Walther. 
 \newblock The solution manifold and {$C^1$}-smoothness for differential equations with state-dependent delay. 
 \newblock {\em J. Differential Equations}, 195(1):46--65, 2003. 
 
 
 \bibitem{9793393} 
 Jinlong Yuan, Changzhi Wu, Kok~Lay Teo, Shuang Zhao, and Lixia Meng. 
 \newblock Perimeter control with state-dependent delays: Optimal control model and computational method. 
 \newblock {\em IEEE Transactions on Intelligent Transportation Systems}, 23(11):20614--20627, 2022. 
 
 
 \bibitem{10.1063/5.0325998} 
 Jiaxuan Zhang, Qunxi Zhu, and Wei Lin. 
 \newblock A general framework for neural delay differential equations with various delay types. 
 \newblock {\em Chaos: An Interdisciplinary Journal of Nonlinear Science}, 36(5):053108, May 2026. 
 
\end{thebibliography}
\end{document}